\documentclass[12pt]{amsart}
\usepackage{amsmath, amsthm, amscd, amssymb, amsfonts, latexsym}
\usepackage{fullpage}
\usepackage{bm,mathdots}
\usepackage{dsfont}
\usepackage[all]{xy}
\usepackage{tikz}
\usepackage{fancybox}
\usepackage{ytableau}

\newcommand{\kommentar}[1]{}

\newcommand{\F}{\mathbb F}

\newcommand{\Z}{\mathbb Z}

\newcommand{\R}{\mathbb R}

\newcommand{\C}{\mathbb C}

\DeclareMathOperator{\sgn}{sgn}

\renewcommand{\pmod}[1]{\,(\mathrm{mod}\,#1)}

\newtheorem{lem}{Lemma}[section]
\newtheorem{prop}[lem]{Proposition}
\newtheorem{thm}[lem]{Theorem}

\newtheorem{cor}[lem]{Corollary}
\newtheorem{conj}[lem]{Conjecture}
\theoremstyle{definition}

\newtheorem{rem}[lem]{Remark}

\newcommand{\vertgeq}{\rotatebox{90}{$\,\leq$}}

\author{Vivian Kuperberg}
\author{Matilde Lal\'in}

\address{Vivian Kuperberg: Department of Mathematics,
School of Mathematical Sciences,
Tel Aviv University,
P.O. Box 39040,
Ramat Aviv, Tel Aviv 69978, Israel }\email{vivkuperberg@tauex.tau.ac.il}

\address{Matilde Lal\'in:  D\'epartement de math\'ematiques et de statistique,
                                    Universit\'e de Montr\'eal,
                                    CP 6128, succ. Centre-ville,
                                     Montreal, QC H3C 3J7, Canada}\email{matilde.lalin@umontreal.ca}

\thanks{This work is supported by NSF GRFP grant DGE-1656518, the NSF Mathematical Sciences Research Program through the grant DMS-2202128,  the Natural Sciences and Engineering Research Council of Canada, Discovery Grant 355412-2022, the Fonds de recherche du Qu\'ebec - Nature et technologies, Projet de recherche en \'equipe 300951, and NSF FRG Grant 1854398 through the American Institute of Mathematics.}

                                     
 \title{Symplectic  conjectures for sums of divisor functions and explorations of an orthogonal regime}

\begin{document}

\begin{abstract}
In \cite{KuperbergLalin}, the authors studied the mean-square of certain sums of the divisor function $d_k(f)$ over the function field $\mathbb{F}_q[T]$ in the limit as $q \to \infty$ and related these sums to integrals over the ensemble of symplectic matrices, along similar lines as previous work of Keating, Rodgers, Roditty-Gershon and Rudnick \cite{KR3} for unitary matrices. We present an analogous problem yielding an integral over the ensemble of orthogonal matrices and pursue a more detailed study of both the symplectic and orthogonal matrix integrals, relating them to symmetric function theory. The function field results lead to conjectures concerning analogous questions over number fields. 
\end{abstract}

\maketitle

\section{Introduction}
The goal of this article is to study the connection between arithmetic sums involving the divisor function and integrals over the ensembles of unitary symplectic and orthogonal matrices. Consider the finite field $\F_q$, where $q$ is an odd prime power.  For a monic polynomial $f \in \F_q[T]$, define the divisor function $d_k(f)$, analogously to the number field case, via
\[d_k(f):=\# \{(f_1,\dots,f_k)\, : \, f=f_1\cdots f_k, \, f_j \mbox{ monic}\}.\] 

In \cite{KR3}, Keating, Rodgers, Roditty-Gershon, and Rudnick consider the distribution of 
\[\mathcal{S}_{d_k;n;Q}(A):= \sum_{\substack{f \,\text{monic}, \deg(f)=n\\f\equiv A \pmod{Q}}}d_k(f).\]
For $Q$ a square-free polynomial and $n\leq k(\deg (Q)-1)$, they prove in \cite[Theorem 3.1]{KR3}
that  the variance of $\mathcal{S}_{d_k;n;Q}$, defined by 
\[\mathrm{Var}(\mathcal{S}_{d_k;n;Q}):=\frac{1}{\Phi(Q)} \sum_{\substack{A\pmod{Q}\\\mathrm{gcd}(A,Q)=1}} \left|\mathcal{S}_{d_k;n;Q}(A)-\left \langle \mathcal{S}_{d_k;n;Q}\right\rangle  \right|^2,\]
is given by 
 \begin{equation}\label{eq:intuni}
 \lim_{q\rightarrow \infty}\frac{\mathrm{Var}(\mathcal{S}_{d_k;n;Q})}{q^n/|Q|}=\int_{\mathrm{U}(\deg(Q)-1)}\Big| \sum_{\substack{j_1+\cdots+j_k=n\\0\leq j_1,\dots,j_k\leq \deg(Q)-1}}\mathrm{Sc}_{j_1}(U)\cdots \mathrm{Sc}_{j_k}(U)\big|^2 \mathrm{d}U,\end{equation}
 where the integral ranges over complex unitary matrices of dimension $\deg(Q)-1$, and the $\mathrm{Sc}_j(U)$ are the secular coefficients, defined for a $N \times N$ matrix $U$ via 
\[\det (I+Ux)=\sum_{j=0}^N \mathrm{Sc}_j(U)x^j.\]

Letting $c=\frac{m}{N} \in [0,k]$, Keating et al. further prove in \cite[Theorem 1.5]{KR3} that  
 \[\int_{\mathrm{U}(N)}\Big| \sum_{\substack{j_1+\cdots+j_k=m\\0\leq j_1,\dots,j_k\leq N}}\mathrm{Sc}_{j_1}(U)\cdots \mathrm{Sc}_{j_k}(U)\Big|^2 \mathrm{d}U=\gamma_k(c)N^{k^2-1}+O_k(N^{k^2-2}),\]
 where
 \begin{equation}\label{eq:gamma}
 \gamma_k(c)=\frac{1}{k!G(1+k)^2}\int_{[0,1]^k}\delta_c(w_1+\cdots +w_k) \prod_{i<j} (w_i-w_j)^2d^kw.
 \end{equation}
Here $\delta_c(w)=\delta(w-c)$ is the delta distribution translated by $c$, and $G$ is the Barnes $G$-function, defined for positive integers $k$ as $G(1+k)=1!\cdot 2!\cdots (k-1)!$. 

Keating at al. study various properties of $\gamma_k(c)$, proving in particular that it is a piecewise polynomial function of $c$. More specifically, $\gamma_k(c)$ is a fixed polynomial for $r\leq c <r+1$ (with $r$ an integer), and each time the value of $c$ passes through an integer, it becomes a different polynomial.

In the integer setting, let $Q\in \Z$ be prime and define for $A\in \Z$,
\[\mathcal{S}_{d_k;X;Q}(A)=\sum_{\substack{n\leq X\\n\equiv A \pmod{Q}}}d_k(n),\]
where $d_k(n)$ is the number of ways of writing 
$n$ as a product of $k$ positive integers. The variance is given by 
\[\mathrm{Var}(\mathcal{S}_{d_k;X;Q}):=\frac{1}{\Phi(Q)} \sum_{\substack{A\pmod{Q}\\\mathrm{gcd}(A,Q)=1}} \left|\mathcal{S}_{d_k;X;Q}(A)-\left \langle \mathcal{S}_{d_k;X;Q}\right\rangle  \right|^2.\]
Based on their result over function fields and their study of the integral in \eqref{eq:intuni},  Keating et al.  conjecture in \cite[Conjecture 3.3]{KR3}  that for $Q^{1+\varepsilon} <X<Q^{k-\varepsilon}$, as $X \rightarrow \infty$, 
\begin{equation}\label{conj:3.3}
\mathrm{Var}(\mathcal{S}_{d_k;X;Q}) \sim \frac{X}{Q} a_k \gamma_k \left( \frac{\log X}{\log   Q}\right) (\log Q)^{k^2-1},
\end{equation}
 where $a_k$ is given by 
\[a_k=\prod_p \left( \Big(1-\frac{1}{p}\Big)^{(k-1)^2}\sum_{j=0}^{k-1}
\binom{k-1}{j}^2\frac{1}{p^j}\right)\]
 and $\gamma_k$ is given by \eqref{eq:gamma}. This conjecture is consistent with work of Lester \cite{Lester}, and Bettin and Conrey \cite{Bettin-Conrey} also showed that the Shifted moments conjecture for the Riemann zeta function implies that the main term of the variance of $\Delta_k(x,H)$ has coefficient $a_k\gamma_k(c)$, as predicted by \eqref{conj:3.3}. Properties of $\gamma_k(c)$ have been further studied, among other places, in \cite{BGR}.

In \cite{KuperbergLalin}, the authors considered the distribution of $d_k(f)$ over $\F_q[T]$ when restricted to quadratic residues modulo an irreducible polynomial $P$. In other words, we defined
\[\mathcal{S}^S_{d_k,n}(P):=\sum_{\substack{f \,\text{monic}, \deg(f)=n  \\f\equiv \square \pmod{P}\\P\nmid f}} d_k(f),\]
where $P$ is a monic irreducible polynomial of degree $2g+1$.
They prove in \cite[Theorem 1.1]{KuperbergLalin} that for  $n\leq 2gk$,  as $q \rightarrow \infty$,
 \begin{align} \label{eq:intsym}
 \frac{1}{\# \mathcal{P}_{2g+1}}\sum_{P\in \mathcal{P}_{2g+1}} \Big|\mathcal{S}^S_{d_k,n}(P)- \frac{1}{2}\sum_{\substack{f\in \mathcal{M}_n \\P\nmid f}} d_k(f)\Big|^2
 \sim & \frac{q^n}{4} I_{d_k,2}^S(n; g),
\end{align}
where
\begin{equation} \label{symplecticintegral}
I_{d_k,2}^S(n;N):=\int_{\mathrm{Sp}(2N)} \Big|\sum_{\substack{j_1+\cdots+j_k=n\\0\leq j_1,\dots,j_k \leq 2N}}\mathrm{Sc}_{j_1}(U)\cdots \mathrm{Sc}_{j_k}(U)\Big|^2 \mathrm{d}U.
\end{equation}
This result, while analogous to formula \eqref{eq:intuni}, involves an integral over the set of unitary symplectic matrices, and crucially uses a monodromy argument of Katz. In order to interpret \eqref{eq:intsym}, one must understand the integral $I_{d_k,2}^S(n;N)$, which we study extensively in this paper. 

As in the case of formula \eqref{eq:intuni}, $I_{d_k,2}^S(n;N)$ is a piecewise polynomial function with transition points whenever $\frac{n}{N}$ passes an integer. In order to better understand $I_{d_k,2}^S(n;N)$, we have the following result about its behavior when $n$ is of the same order as $N$, which we obtain by further developing some results from Medjedovic's MSc thesis \cite{Andy} over symmetric functions.

\begin{thm}  \label{thm:degreeEhrharts}
  Let $c=\frac{a}{b}$ be a fixed rational number and $k$ be a fixed integer. If $2N$ is a multiple of $b$, then $I_{d_k,2}^S(c2N;N)$ is a polynomial of degree ${2k^2+k-2}$ in $N$.
\end{thm}

This is done by expressing a generating function involving $I_{d_k,2}^S(n;N)$ in terms of a sum of Schur functions over certain even partitions. This allows us to interpret $I_{d_k,2}^S(n;N)$ as a function counting points inside a polytope, and combining this with Ehrhart theory.

Of particular interest is the leading coefficient of $I_{d_k,2}^S(c2N;N)$, which is a function in $c$. Let $\gamma_{d_k,2}^S(c)$ denote the leading coefficient of $I_{d_k,2}^S(c2N;N)$. By using complex analysis techniques to analyze $I_{d_k,2}^S(n,N)$, we get a description of $\gamma_{d_k,2}^S$ as a piecewise polynomial function of degree at most $2k^2+k-2$ for any real number, made precise in the following result. 

\begin{thm}\label{thm:complexsymplectic} The function $\gamma_{d_k,2}^S(c)$ is given by
\[ \gamma_{d_k,2}^S(c) = \sum_{\substack{0\leq b\leq c \\ 0 \le a \le 2c-b}} g_{a,b}^S(c),\]
and each $g_{a,b}^S(t)$ is a polynomial of degree $2k^2 + k - 2$.
\end{thm}
\begin{rem}
This result exhibits the transition points of the piecewise polynomial $\gamma_{d_k,2}^S(c)$. In particular, whenever $c$ passes an integer or a half-integer, the number of terms in the sum in Theorem \ref{thm:complexsymplectic} increases, since the number of pairs $(a,b)$ with $0 \le b \le c$ and $0 \le a \le 2c-b$ increases. These are the transition points.
\end{rem}

Both Theorems \ref{thm:degreeEhrharts} and \ref{thm:complexsymplectic} indicate that the degree of $\gamma_{d_k,2}^S$ is at most $2k^2 + k - 2$. We expect equality to hold, but our results do not eliminate the possibility that the degree could be smaller due to possible cancellation of the main coefficients of $g_{a,b}^S(c)$. The situation is nearly identical to that of the unitary case, where in \cite{KR3}, Theorem 1.6 shows that $\gamma_k(c)$ is a sum of polynomials of degree $k^2-1$ whose main coefficients may cancel. It would be interesting to investigate this question further.

\begin{conj}\label{conj:deg-gamma-symplectic-equality}
The degree of $\gamma_{d_k,2}^S(c)$ is $2k^2 + k - 2$.
\end{conj}

In the second half of this paper, we consider quantities related to orthogonal matrix integrals. We will discuss number-theoretic problems which lead us to consider the integral
\begin{equation*} 
I_{d_k,2}^O(n;N):=\int_{\mathrm{O}(2N+1)} \Big|\sum_{\substack{j_1+\cdots+j_k=n\\0\leq j_1,\dots,j_k \leq 2N+1}}\mathrm{Sc}_{j_1}(U)\cdots \mathrm{Sc}_{j_k}(U)\Big|^2 \mathrm{d}U,
\end{equation*}
as well as the main coefficient  $\gamma_{d_k,2}^O$ of the highest power of $N$.

We proceed to further study $I_{d_k,2}^O$ and $\gamma_{d_k,2}^O$ and prove results analogous to our results in the symplectic case. 
More precisely, we prove the following results.
\begin{thm}  \label{thm:degreeEhrharto}
  Let $c=\frac{a}{b}$ be a fixed rational number and $k$ be a fixed integer. If $2N+1$ is a multiple of $b$, then $I_{d_k,2}^O(c(2N+1);N)$ is a polynomial of degree ${2k^2-k-2}$ in $N$.
\end{thm}
Note that the proof of Theorem \ref{thm:degreeEhrharto}, presented in Section \ref{sec:symmetric-orthogonal}, also applies when to the integral analogous to $I_{d_k,2}^O(n;N)$ but defined over the even-dimensional group $\mathrm{O}(2N)$.

\begin{thm}\label{thm:complexorthogonal} The constant $\gamma_{d_k,2}^O(c)$ is given by
\[ \gamma_{d_k,2}^O(c) =  \sum_{\substack{0\leq b\leq c \\ 0 \le a \le 2c-b}} g_{a,b}^O(c),\]
and each $g_{a,b}^O(t)$ is a polynomial of degree  $2k^2-k-2$.
\end{thm}
\begin{rem}
Just as in the symplectic case, this result exhibits the transition points of the piecewise polynomial $\gamma_{d_k,2}^O(c)$ whenever $c$ passes an integer or a half-integer.
\end{rem}
Theorems \ref{thm:degreeEhrharto} and \ref{thm:complexorthogonal} imply that
\[\deg\left(\gamma_{d_k,2}^O\right)\le 2k^2-k-2,\]
and suggest that equality should hold.
\begin{conj}
The degree of $\gamma_{d_k,2}^O(c)$ is $2k^2 - k - 2$.
\end{conj}

Based on our understanding of $\gamma_{d_k,2}^S$, we formulate conjectures concerning questions on the distribution of $d_k$ over number fields that are analogous to the Gaussian integer setting explored over the function fields. 

To begin with we have the following conjecture, which is analogous to the conjecture given by \eqref{conj:3.3}.
\begin{conj} \label{conj:symp-square} Let $p$ be a prime and define

\[\mathcal{S}^S_{d_k;x}(p)=\sum_{\substack{n\leq x  \\n\equiv \square \pmod{p} \\ p \nmid n}}d_k(n).\]
Consider the variance defined by
\begin{align*}
\mathrm{Var}_{p\in [y,2y]}\left(\mathcal{S}^S_{d_k;x}\right) &:= \frac{1}{y} \sum_{y < p \le 2y} \log p \Big(\mathcal{S}^S_{d_k;x}(p) - \left\langle \mathcal{S}^S_{d_k;x}\right\rangle_p\Big)^2,
\end{align*}
where $\langle \cdot \rangle_p$ denotes a refined average given by
\begin{align*}
\left\langle \mathcal{S}^S_{d_k;x}\right \rangle_p &=  \frac 12 \sum_{\substack{n \le x \\ p\nmid n}} d_k(n).
\end{align*}

Then, for $x^{1/k+\varepsilon}\leq y$, we have,
\begin{equation}\label{eq:conj}
\mathrm{Var}_{p\in [y,2y]}\left(\mathcal{S}^S_{d_k;x}\right)
 \sim a_k^{S}(\mathcal S) \frac{x}{4}\gamma_{d_k,2}^S\left(\frac{\log x}{\log y}\right) (\log y)^{2k^2+k-2},
\end{equation}
where $a_k^S(\mathcal S)$ is a certain arithmetic constant dependent only on $k$ and $\gamma_{d_k,2}^S(c)$ is a piecewise polynomial of degree $2k^2+k-2$ given by \eqref{defi:gammaS}.
\end{conj}

\begin{rem}
The statement of this conjecture differs slightly from the published version in the definition the variance, which has now been defined relative to the more refined average value $\langle \mathcal S^S_{d_k;x}\rangle_p$, which is dependent on the prime $p$.
\end{rem}

For $p \nmid n$, the constraint that $n \equiv \square \bmod p$ is equivalent to multiplying by $\frac{1 + \chi_p(n)}{2}$ in the sum. One could alternatively consider the sum without the condition 
$p \nmid n$, namely, where the  $n$  that are divisible by $p$ are counted with weight $\frac{1}{2}$. One could also study the related problem of the variance of $\sum_{n \le x} d_k(n)\chi_p(n)$ when ranging over primes $p$. This problem is perhaps more direct from the perspective of the $L$-functions involved, but is less immediately accessible as a number theoretic question. For this reason and for consistency with \cite{KuperbergLalin}, we have stated Conjecture \ref{conj:symp-square} for a restricted sum over quadratic residues mod $p$.

Constants such as $a_k^S(\mathcal S)$ can be derived by following the recipe of \cite{CFKRS} or by considering $y$ large respect to $x$, in which case the main term is expected to come from the diagonal contribution, and hence arises from a residue computation. However, Conjecture \ref{conj:symp-square} and the constant $a_k^S(\mathcal S)$ are somewhat more subtle because of the constraint that the variance is running only over primes, rather than a wider average. We can also make an analogous conjecture for the variance of the divisor function weighted by quadratic characters where the variance is taken over all positive fundamental discriminants, and not just over primes. In this setting, we conjecture the following.

\begin{conj} \label{conj:symp-square-fund-discs} For a positive fundamental discriminant $m$, define 
\[\mathcal{T}^S_{d_k;x}(m)=\sum_{\substack{n\leq x  \\(m,n) = 1}}d_k(n)\chi_m(n),\]
where $\chi_m$ is the primitive quadratic character mod $m$. 
Let $y > 0$. Define the variance of $\mathcal T^S_{d_k;x}(m)$ in $[y,2y]$ by 
\begin{align*}
\mathrm{Var}_{m\in [y,2y]}\left(\mathcal{T}^S_{d_k;x}\right) &:= \mathbb E^*_{y < m \le 2y} \Big(\mathcal T^S_{d_k;x}(m)\Big)^2,
\end{align*}
where the expectation $\mathbb E^*$ is taken over positive fundamental discriminants.

Then, for $x^{1/k+\varepsilon}\leq y$, we have,
\begin{equation}\label{eq:conj-fund-discs}
\mathrm{Var}_{y < m \le 2y}\left(\mathcal{T}^S_{d_k;x}\right)
 \sim a_k^S(\mathcal T) x\gamma_{d_k,2}^S\left(\frac{\log x}{\log y}\right) (\log y)^{2k^2+k-2},
\end{equation}
where $a_k^S(\mathcal T)$ is an arithmetic constant given by
\begin{equation*}
a_k^S(\mathcal T) = \prod_p \left(1-\frac 1p\right)^{k(2k+1)} \left(\frac 1{p+1}\left(1 + \frac p2\left(\left(1 + \frac 1{\sqrt p}\right)^{-2k} + \left(1 - \frac 1{\sqrt p}\right)^{-2k}\right)\right)\right),
\end{equation*}
and $\gamma_{d_k,2}^S(c)$ is a piecewise polynomial of degree $2k^2+k-2$ given by \eqref{defi:gammaS}.
\end{conj}

The constant $a_k^S(\mathcal T)$ can be derived from the recipe of \cite{CFKRS} or, as discussed in the case of $a_k^S(\mathcal S)$, by considering $y$ large respect to $x$, and extracting the diagonal contribution.
It should be related to the constant $a_k^S(\mathcal S)$ appearing in Conjecture \ref{conj:symp-square}. The authors plan to analyze these constants and present their heuristic derivations in forthcoming work.
          
A crucial piece of information for writing formula \eqref{eq:conj} has to do with the power of $\log y$. This comes from the passage from $I_{d_k,2}^S(n,N)$ to extracting the main coefficient $\gamma_{d_k,2}^S$ of the highest power of $N$. Thus, this power is exactly the degree of the piecewise polynomial function $\gamma_{d_k,2}^S$, which we prove is at most $2k^2 + k - 2$ and according to Conjecture \ref{conj:deg-gamma-symplectic-equality} exactly $2k^2 + k - 2$.

In \cite{KuperbergLalin}, the authors made use of a model of the Gaussian integers in the function field context which was considered by Rudnick and Waxman in \cite{Rudnick-Waxman} and initially developed by Bary-Soroker, Smilansky, and Wolf in \cite{BSSW}. The authors considered the variance of 
\[\mathcal{N}^S_{d_\ell,k,n}(u)=\sum_{\substack{f \,\text{monic}, \deg(f)=n\\ f(0)\not =0\\U(f)\in \mathrm{Sect}(u,k)}} d_\ell(f),\]
where the sum is taken over monic polynomials of fixed degree with certain condition (see \eqref{eq:sector}) that can be interpreted as the function field analogue of having the argument of a complex number lying in certain specific sector of the unit circle. 
 
The authors proved in \cite[Theorem 1.2]{KuperbergLalin} that for $n \leq \ell (2\kappa-2)$ with $\kappa=\left\lfloor \frac{k}{2}\right \rfloor\geq 4$,  as $q\rightarrow \infty$, 
 \[\langle \mathcal{N}^S_{d_\ell,k,n}\rangle\sim q^{n-\kappa}\binom{\ell+n-1}{\ell-1},\]
  and
\begin{equation}\label{thm:sym-RW}\frac{1}{q^\kappa}\sum_{u \in \mathbb{S}_k^1} \left|\mathcal{N}^S_{d_\ell,k,n}(u)-\langle \mathcal{N}^S_{d_\ell,k,n}\rangle\right|^2\sim\frac{q^n}{q^{\kappa}}\int_{\mathrm{Sp}(2\kappa-2)}\Big| \sum_{\substack{j_1+\cdots+j_\ell=n\\0\leq j_1,\dots,j_\ell \leq 2\kappa-2}}\mathrm{Sc}_{j_1}(U)\cdots \mathrm{Sc}_{j_\ell}(U)\Big|^2 \mathrm{d}U.\end{equation}

We can also use our understanding of $\gamma_{d_k,2}^S(c)$ to formulate a conjecture for an analogous problem in the integers.
Following Rudnick and Waxman \cite{Rudnick-Waxman}, for an ideal $\mathfrak{a}=(\alpha)$ of $\Z[i]$ we associate a direction vector $u(\mathfrak{a})=u(\alpha):=\left(\frac{\alpha}{\overline{\alpha}}\right)^2$ in the unit circle and write 
$u(\mathfrak{a})=e^{i4\theta_\mathfrak{a}}$. Since the generator of the ideal is defined up to multiplication by a unit $\{\pm 1, \pm i\}$, the direction vector is well defined on ideals, while the angle $\theta_\mathfrak{a}$ is well-defined modulo $\pi/2$.  

For a given  $\theta$ consider the neighborhood $I_K(\theta)=[\theta-\frac{\pi}{4K},\theta+\frac{\pi}{4K}]$. It is then natural to consider the ideals $\mathfrak{a}$ such that $\theta_\mathfrak{a} \in I_K$, leading to the study of 
\[\mathcal{N}^S_{d_\ell,K;x}(\theta)=\sum_{\substack{\mathfrak{a} \, \text{ideal} \\N(\mathfrak{a})\leq x\\ 
\theta_\mathfrak{a} \in I_K(\theta) }} d_\ell(\mathfrak{a}).\]
Its average is given by 
\begin{align*}
 \left\langle \mathcal{N}^S_{d_\ell,K;x} \right\rangle:= &\frac 2{\pi} \int_0^{\pi/2}  \sum_{\substack{\mathfrak{a} \, \text{ideal} \\N(\mathfrak{a})\leq x\\ \theta_\mathfrak{a} \in I_K(\theta) }} d_\ell(\mathfrak{a}) d\theta
 =\sum_{\substack{\mathfrak{a}\, \text{ideal}\\ N(\mathfrak{a})\leq x}} d_\ell(\mathfrak{a})
 \frac 2{\pi} \int_0^{\pi/2} \mathds{1}_{\theta_\mathfrak{a}\in I_K(\theta)} \mathrm{d}\theta
 = \frac{1}{K}\sum_{\substack{\mathfrak{a}\, \text{ideal}\\ N(\mathfrak{a})\leq x}} d_\ell(\mathfrak{a}),
\end{align*}
and its variance is 
\begin{align*}
 \mathrm{Var}\left(\mathcal{N}^S_{d_\ell,K;x}\right) := \frac 2{\pi} \int_0^{\pi/2} \Big(\sum_{\substack{\mathfrak{a} \, \text{ideal} \\N(\mathfrak{a})\leq x\\ \theta_\mathfrak{a} \in I_K(\theta) }} d_\ell(\mathfrak{a})- \left\langle
 \mathcal{N}^S_{d_\ell,K;x} \right\rangle\big)^2 \mathrm{d}\theta.
\end{align*}

We get the following conjecture by replacing $q^n$ by $x$ and $q^{\kappa}$ by $K$.

\begin{conj}\label{conj:symp-rudnickwaxman}
 Let $x \leq K^{\ell}$. Then as $x \to \infty$, the variance $\mathrm{Var}\left(\mathcal N^S_{d_\ell,K;x}\right)$ is given by 
\[\mathrm{Var}\left(\mathcal{N}^S_{d_\ell,K;x}\right)\sim a^S_\ell(\mathcal N) \frac{x}{K}
 \gamma_{d_\ell,2}^S\left(\frac{\log x}{2\log K}\right) (2\log K)^{2\ell^2+\ell-2},\]
 where
 \begin{equation*}
a^S_\ell(\mathcal N) := \left(\frac{\pi}{8b^2}\right)^{\ell(2\ell - 1)} \frac{(2+\sqrt{2})^{2\ell} + (2-\sqrt{2})^{2\ell}}{2^{\binom{2\ell + 1}{2}+1}} \prod_{p \equiv 1 \bmod 4} \left(1-\frac 1p\right)^{4\ell^2} \sum_{j=0}^\infty \binom{j+2\ell-1}{j}^2 \frac 1{p^j},
 \end{equation*}
 for $b$ the Landau--Ramanujan constant given by 
 \begin{equation*}
b = \frac 1{\sqrt 2} \prod_{p \equiv 3\bmod 4} \left(1-\frac 1{p^2}\right)^{-\tfrac 12}.
 \end{equation*}
 \end{conj}
The constant $a_\ell^S(\mathcal N)$ can be derived by following the recipe of \cite{CFKRS} or, as discussed in the case of $a_k^S(\mathcal S)$, by considering $K$ large respect to $x$, and extracting the diagonal contribution.  The authors plan on addressing this and similar derivations in forthcoming work. 

Finally, we consider a number-theoretic problem, in both the function field and integer setting, whose behavior is connected to an orthogonal matrix integral. Let
\[\mathcal N_{d_\ell,k,n}^O(u) := \sum_{\substack{f \in \mathcal M_{n} \\ f(0) \ne 0 \\ U(f) \in \mathrm{Sect}(u,k)}} d_\ell(f) \left(\frac{1 + \chi_2(f)}{2}\right),\]
where $\chi_2$ is the quadratic character over $\F_q[T]$ defined by $\chi_2(f):=\chi_2(f(0))$, where for $x \in \F_q$,  $\chi_2(x):=1$ if $x$ is a non-zero square, $-1$ if $x$ is a non square, and $0$ if $x=0$. As we show in Lemma \ref{lem:meanNo}, the average value of $\mathcal N_{d_\ell,k,n}^O(u)$ is
\begin{equation*}
\langle \mathcal{N}^O_{d_\ell,k,n} \rangle :=\frac{1}{q^\kappa} \sum_{u \in \mathbb{S}_k^1}  \mathcal{N}^O_{d_\ell,k,n}(u)
=\frac{q^{n-\kappa}}{2} \binom{\ell+n-1}{\ell-1} +O(q^{n-\kappa-1}).
\end{equation*}
We estimate the variance of $\mathcal N_{d_\ell,k,n}^O(u)$ in the following result.
\begin{thm}\label{thm:ortho-intro} Let $n\leq \ell (2\kappa-1)$ with $\kappa =\lfloor \frac{k}{2}\rfloor \geq 3$. As $q \rightarrow \infty$,   
\[\frac{1}{q^\kappa}\sum_{u \in \mathbb{S}_k^1} \left|\mathcal{N}^O_{d_\ell,k,n}(u)-\langle \mathcal{N}^O_{d_\ell,k,n}(u)\rangle_S\right|^2 \sim  \frac{q^n}{4q^{\kappa}} \int_{\mathrm{O}(2\kappa - 1)} \Big| \sum_{\substack{j_1 + \cdots + j_\ell = n \\ 0 \le j_1, \dots, j_\ell \le 2\kappa - 1}} \mathrm{Sc}_{j_1}(U) \cdots \mathrm{Sc}_{j_\ell}(U)\Big|^2 \mathrm{d}U,\]
where $\langle \mathcal N_{d_\ell,k,n}^O(u)\rangle_S$ is a refined average given by 
\begin{equation*}
\langle \mathcal N_{d_\ell,k,n}^O(u)\rangle_S := \frac 12 \sum_{\substack{f \in \mathcal M_n \\ f(0) \ne 0 \\ U(f) \in \mathrm{Sect}(u,k)}} d_\ell(f).
\end{equation*}
\end{thm}
Both Theorem \ref{thm:ortho-intro} as well as the result \eqref{thm:sym-RW} from \cite[Theorem 1.2]{KuperbergLalin} rely on monodromy theorems due to Katz \cite{Katz}. Surprisingly, straightforward rational analogs of $\mathcal N_{d_\ell,k,n}^O(u)$ are related to families of $L$-functions that do not have orthogonal symmetry, so Theorem \ref{thm:ortho-intro} does not naturally lead to a corresponding conjecture for its rational analog. However, we expect the orthogonal integral $I_{d_k,2}^O(c(2N+1);N)$ and the quantity $\gamma_{d_k,2}^O(c)$ to appear in other problems involving orthogonal families of rational $L$-functions; see Section \ref{sec:conjectures} for more discussion. 

More generally, for any  family $\mathcal F$ of rational $L$-functions with symmetry type $G$, we expect
\[\mathbb E_{f \in \mathcal F, \mathrm{conductor}(f) \le y} \left|\sum_{m \le x} d_k(m)f(m)\right|^2\]
to be asymptotically proportional to  
\[I_{d_k,2}^G
:=\int_{G} \Big|\sum_{\substack{
}}\mathrm{Sc}_{j_1}(U)\cdots \mathrm{Sc}_{j_k}(U)\Big|^2 \mathrm{d}U,\]
where the sum in the integrand is taken over $j_1, \dots, j_k$ with $j_1 + \cdots + j_k$ constant. The proportionality constant may be extracted by considering $y$ sufficiently large with respect to $x$, as then only the diagonal contribution to $\mathbb E_{f \in \mathcal F, \mathrm{conductor}(f) \le y} \left|\sum_{m \le x} d_k(m)f(m)\right|^2$ is relevant. 

We also remark that in the symplectic and orthogonal cases one could also consider 
\[\mathbb E_{f \in \mathcal F, \mathrm{conductor}(f) \le y} \sum_{m \le x} d_k(m)f(m)\]
without the absolute value. This is analogous to the study of moments of $L$-functions versus the Riemann zeta function. In the Riemann zeta function case, which has unitary symmetry, one considers the $2k$th  moments of the absolute value, since the $k$th moments without the absolute value are highly oscillatory.  However, in the $L$-functions cases with symplectic or ortogonal symmetry, it is possible to study the $k$th moments without taking absolute values. We keep the absolute values in this article to better compare with the original number theoretic questions.

While the methods in this article are quite close to the ones employed in \cite{KR3}, they are technically more involved. One central difference is that in the context of random matrix theory, it is much less natural to consider integrals of squares of functions in the symplectic and orthogonal cases than it is in the unitary case, where the eigenvalues are complex and the square of the absolute value is essentially unavoidable. However, the problems that we consider are the variances of various arithmetic functions, so the squares in the integral are inherent to our setting. The presence of the square brings considerable difficulty in the study of $\gamma_{d_k,2}$ both in the symplectic and the orthogonal cases.  The symmetric function theory in the unitary case considered in \cite{KR3} requires the consideration  of semi-standard Young tableaux arising from a single rectangular Ferrer diagram, while here we must consider a sum including more general shapes. The orthogonal case is particularly more involved, and it requires developing machinery for $\mathrm{O}(2N+1)$ which we have not found in the literature.  Also, at the level of the complex integrals we find considerable challenges, again coming from the square introduced in the symplectic and orthogonal cases.

This article is organized as follows. In Section \ref{sec:determinant} we develop some results about determinants that will be useful for computations with symmetric functions. In Sections \ref{sec:symmetric-symplectic} and \ref{sec:symmetric-orthogonal} we express generating functions involving $I_{d_k,2}^S(n;N)$ and $I_{d_k,2}^O(n;N)$,  respectively, in terms of sums of Schur functions and combine this with Ehrhart theory to deduce the degrees of $\gamma_{d_k,2}^S$ and $\gamma_{d_k,2}^O$ and prove Theorems \ref{thm:degreeEhrharts} and \ref{thm:degreeEhrharto}. We combine the results of those sections in Sections \ref{sec:determinant-symplectic} and  \ref{sec:determinant-orthogonal} to obtain some specific formulas for $\gamma_{d_k,2}^S$ and $\gamma_{d_k,2}^O$. We treat the generating functions of $I_{d_k,2}^S(n;N)$ and $I_{d_k,2}^O(n;N)$ as complex integrals in Sections \ref{sec:complex-symplectic} and \ref{sec:complex-orthogonal}, providing further insight on the shapes of $\gamma_{d_k,2}^S$ and $\gamma_{d_k,2}^O$ and proving Theorems \ref{thm:complexsymplectic} and \ref{thm:complexorthogonal}. Section \ref{sec:RW} includes the function field setting for the Gaussian integers, as well as the proof of Theorem \ref{thm:ortho-intro}. Finally, in Section \ref{sec:conjectures} we discuss the motivation for Conjectures \ref{conj:symp-square} and Conjecture \ref{conj:symp-rudnickwaxman}.

\medskip
\noindent {\bf Acknowledgments:} We are thankful to the referee for helpful corrections and suggestions.   We are grateful to Siegfred Baluyot, Brian Conrey, Ofir Gorodetsky, Nicholas Katz, Jonathan Keating, Andean Medjedovic, Brad Rodgers, Michael Rubinstein, Zeev Rudnick, and Kannan Soundararajan for many helpful discussions. Part of this work was completed during conferences at the American Institute of Mathematics, the Institute for Advanced Study, and the Mathematisches Forschungsinstitut Oberwolfach and we are thankful for their hospitality.

\section{Some results on determinants}\label{sec:determinant}
This section covers some results on determinants that will be needed later. 

Let $p_j(x) \in \C[x]$ for $j=1,\dots, 2k$ and consider the $2k$-variable polynomial given by 
\begin{equation*}
P(x_1,\dots, x_{2k})=\det_{1\leq i, j \leq 2k}[p_j(x_{i})].
\end{equation*}
This is an alternating polynomial, and thus divisible by the Vandermonde determinant $\Delta(x_1,\dots,x_{2k}):=\prod_{1\leq i <j\leq 2k} (x_i-x_j)$. We are interested in finding $\frac{P(x,\dots,x,y,\dots,y)}{\Delta(x,\dots,x,y,\dots,y)}$, or more precisely, the limit of $\frac{P(x_1, \dots, x_{2k})}{\Delta(x_1,\dots,x_{2k})}$ as $x_1,\dots, x_{k}\rightarrow x$ and $x_{k+1},\dots, x_{2k}\rightarrow y$.

\begin{lem}\label{lem:superAndy}
We have
 \begin{equation}\label{eq:superAndy}\left.\frac{P(x_1,\dots, x_{2k})}{\Delta(x_1,\dots, x_{2k})}\right|_{(x,\dots,x,y,\dots,y)}=\frac{1}{(y-x)^{k^2}}\det_{\substack{1\leq i_1,i_2 \leq k\\1\leq j \leq 2k}} \left[\begin{array}{c}\frac{1}{(i_1-1)!} \frac{\partial^{i_1-1}}{\partial x^{i_1-1}} p_j(x)\\ \\ \frac{1}{(i_2-1)!} \frac{\partial^{i_2-1}}{\partial y^{i_2-1}} p_j(y)\end{array} \right],
 \end{equation}
 where the vector $(x,\dots,x,y,\dots,y)$ in \eqref{eq:superAndy} has  $k$ coordinates equal to $x$ and $k$ coordinates equal to $y$. 
\end{lem}
The proof of this lemma follows from the same ideas as \cite[Lemma 1]{Andy}, which considers the case of a $k\times k$ matrix and $\frac{P(x,\dots,x)}{\Delta(x,\dots,x)}$.

\begin{proof}
First consider the case 
 \begin{equation}\label{eq:Psum}
 P(x_1,\dots, x_{2k})=\det_{1\leq i, j \leq 2k}[x_{i}^{a_j}] = \sum_{\sigma \in \mathbb{S}_{2k}} \sgn(\sigma) \prod_{i=1}^{2k} x_i ^{a_{\sigma(i)}},
 \end{equation}
where the $a_j$ are non-negative integers. By applying L'H\^opital's rule, the limit that we seek to compute becomes
 \begin{align*}\lim_{\substack{x_i\rightarrow x, \, 1\leq i \leq k\\x_j\rightarrow y, \, k+1\leq j \leq 2k}} \frac{P(x_1,\dots, x_{2k})}{\Delta(x_1,\dots, x_{2k})}= &\left[ \frac{1}{(1!\cdots (k-1)!)^2}\frac{1}{ \prod_{1\leq i \leq k}\prod_{k+1\leq j \leq 2k} (x_i-x_j)}\right.\\&\left. \left. \times (-1)^\frac{(k-1)k}{2}\frac{\partial^{k-1}}{\partial x_{2k}^{k-1}}\cdots \frac{\partial}{\partial x_{k+2}}\frac{\partial^{k-1}}{\partial x_{k}^{k-1}}\cdots \frac{\partial}{\partial x_2}P(x_1,\dots,x_{2k})\right]\right|_{(x,\dots,x,y,\dots,y)}.
 \end{align*}
Introducing the expansion \eqref{eq:Psum} in the above identity yields
\begin{align*}&\lim_{\substack{x_i\rightarrow x, \, 1\leq i \leq k\\x_j\rightarrow y, \, k+1\leq j \leq 2k}} \frac{P(x_1,\dots, x_{2k})}{\Delta(x_1,\dots, x_{2k})}=\left[ \frac{1}{(1!\cdots (k-1)!)^2}\frac{1}{ \prod_{1\leq i \leq k}\prod_{k+1\leq j \leq 2k} (x_i-x_j)}\right. \\ & \left. \left. \times  \frac{\partial^{k-1}}{\partial x_{2k}^{k-1}}\cdots \frac{\partial}{\partial x_{k+2}}\frac{\partial^{k-1}}{\partial x_{k}^{k-1}}\cdots \frac{\partial}{\partial x_2}\sum_{\sigma \in \mathbb{S}_{2k}} \sgn(\sigma) \prod_{i=1}^{2k} x_i ^{a_{\sigma(i)}}\right] \right|_{(x,\dots,x,y,\dots,y)}\\ 
= & \left.\frac{1}{  (y-x)^{k^2}} \sum_{\sigma \in \mathbb{S}_{2k}} \sgn(\sigma) \prod_{i_1=1}^{k} \binom{a_{\sigma(i_1)}}{i_1-1} x_{i_1}^{a_{\sigma(i_1)}-i_1+1}\prod_{i_2=k+1}^{2k}\binom{a_{\sigma(i_2)}}{i_2-k-1} x_{i_2}^{a_{\sigma(i_2)}-i_2+k+1}\right|_{(x,\dots,x,y,\dots,y)}\\
=&\frac{1}{(y-x)^{k^2}}\det_{\substack{1\leq i_1,i_2 \leq k\\1\leq j \leq 2k}} \left[\begin{array}{c} \frac{1}{(i_1-1)!} \frac{\partial^{i_1-1}}{\partial x^{i_1-1}} x^j\\ \\ \frac{1}{(i_2-1)!} \frac{\partial^{i_2-1}}{\partial y^{i_2-1}} y^j\end{array} \right],
 \end{align*}
as desired. For the general case, apply multi-linearity of the determinant to split it into a sum of determinants involving monomials, and apply the above result to each individual term. The proof is completed by adding the terms together.
 
\end{proof}

The right-hand side of \eqref{eq:superAndy} is a polynomial. It can be computed from a particular case of the generalized Vandermonde matrix in \cite{Kalman}, which we state in Lemma \ref{lem:y-xk2}.
\begin{lem} \label{lem:y-xk2}
For any positive integer $k$,
\[\det_{\substack{1 \le i_1, i_2, \le k \\ 1 \le j \le 2k}} \left[\begin{array}{c}\frac{(j-1)!}{(j-i_1)!} x^{j-i_1}\\ \\ \frac{(j-1)!}{(j-i_2)!} y^{j-i_2} \end{array}\right] = G(1+k) (y-x)^{k^2},\]
where $G(1+k)=0! \cdot 1! \cdots (k-1)!$ is the Barnes $G$-function. 
\end{lem}

Subsequently we will need the following generalization of Lemma \ref{lem:y-xk2}.
\begin{prop}\label{prop:superVivian}
  Let $\alpha_1,\dots, \alpha_{2k}$ be non-negative integers and let 
   \[M = \begin{bmatrix} x^{\alpha_1} & x^{\alpha_2} & \cdots & x^{\alpha_{2k}}\\
  \alpha_1 x^{\alpha_1-1} & \alpha_2 x^{\alpha_2-1} & \cdots & \alpha_{2k} x^{\alpha_{2k}-1}\\
    \alpha_1 (\alpha_1-1) x^{\alpha_1-2} & \alpha_2 (\alpha_2-1) x^{\alpha_2-2} & \cdots & \alpha_{2k}(\alpha_{2k}-1) x^{\alpha_{2k}-2}\\ 
 \vdots & \vdots  & \ddots & \vdots  \\
  \alpha_1\cdots  (\alpha_1-k+2) x^{\alpha_1-k+1} & \alpha_2 \cdots (\alpha_2-k+2) x^{\alpha_2-k+1} & \cdots & \alpha_{2k}\cdots (\alpha_{2k}-k+2) x^{\alpha_{2k}-k+1}\\ 
  y^{\alpha_1} & y^{\alpha_2} & \cdots & y^{\alpha_{2k}}\\
  \alpha_1 y^{\alpha_1-1} & \alpha_2 y^{\alpha_2-1} & \cdots & \alpha_{2k} y^{\alpha_{2k}-1}\\
    \alpha_1 (\alpha_1-1) y^{\alpha_1-2} & \alpha_2 (\alpha_2-1) y^{\alpha_2-2} & \cdots & \alpha_{2k}(\alpha_{2k}-1) y^{\alpha_{2k}-2}\\ 
 \vdots & \vdots  & \ddots & \vdots  \\
  \alpha_1\cdots  (\alpha_1-k+2) y^{\alpha_1-k+1} & \alpha_2 \cdots (\alpha_2-k+2) y^{\alpha_2-k+1} & \cdots & \alpha_{2k}\cdots (\alpha_{2k}-k+2) y^{\alpha_{2k}-k+1}.
\end{bmatrix}\]
  Then $(y-x)^{k^2}$ divides $\det M$. 
 \end{prop}
 \begin{proof}
The first  $k$ rows consist of the successive derivatives of the first row; then the last $k$ rows follow the same structure, but with $y$'s instead of $x$'s. 

Every term in $\det M$ has degree $\sum_{j=1}^{2k}\alpha_j-(k-1)k$ or is identically zero, so $\det M$ must be a homogeneous polynomial of degree $\sum_{j=1}^{2k}\alpha_j-(k-1)k$ (or identically zero). 

Subtracting the $k+1$st row from the first row does not change the determinant, and the difference is a row of terms of the form $x^{\alpha_j}-y^{\alpha_j}$. Each entry of this new row is divisible by $y-x$, so $y-x$ must divide the determinant. Moreover, the same argument applies for any matrix with two rows which have the same polynomial entries, one in the variable $x$ and the other in the variable $y$; any such matrix has determinant divisible by $y-x$. We will proceed by showing that $(y-x)$ also divides each of the first $k^2-1$ derivatives (in $x$) of $\det M$.

For a $k$-tuple $\beta_1, \dots, \beta_k$, let $M_{\beta_1, \dots, \beta_k}$ denote the matrix $M$ where the $i$th row has been differentiated $\beta_i$ times in $x$. So, $M_{0,\dots,0} = M$. If for any $i$, we have $\beta_i \le k-i$, then the $i$th row of $M_{\beta_1, \dots, \beta_k}$ is the $i-1+\beta_i$'th derivative of $\begin{bmatrix}x^{\alpha_1} & x^{\alpha_2}  &\cdots& x^{\alpha_{2k}}\end{bmatrix}$ and thus identical  to the $k + i -1+ \beta_i$th row, but with $x$'s instead of $y$'s. By the argument above, if $\beta_i \le k-i$ for any $i$, then $(y-x)$ divides $\det M_{\beta_1, \dots, \beta_k}$. 
Also, if for any $1\leq i\leq k-1$ we have $\beta_i = \beta_{i+1}+1$, then the $i$th row and the $i+1$st row  will be identical, so that $\det M_{\beta_1, \dots, \beta_k} = 0$. By extension, if for any $1 \le i,j \le k-1$ we have $\beta_i = \beta_{i+j}+j$, then $\det M_{\beta_1, \dots, \beta_k} = 0$. 

In general for an $\ell \times \ell$ matrix $F(x)$ of functions of $x$ whose rows are $r_1(x), \dots, r_\ell(x)$, the derivative of $\det F(x)$ is given by
\[\frac{\mathrm{d}}{\mathrm{d}x} \det F(x) = \sum_{i=1}^\ell \det \begin{bmatrix} r_1(x)\\ \vdots\\ r_{i-1}(x)\\ \frac{\mathrm{d}}{\mathrm{d}x} r_i(x)\\ r_{i+1}(x)\\ \vdots\\ r_\ell(x) \end{bmatrix}.\]
In our case,  the last $k$ rows are constant in $x$ and they do not have to be considered. So the derivative of $\det M_{\beta_1, \dots, \beta_k}$ is
\[\frac{\mathrm{d}}{\mathrm{d}x} \det M_{\beta_1, \dots, \beta_k} = \det M_{\beta_1 + 1, \beta_2, \dots, \beta_k} + \det M_{\beta_1, \beta_2 + 1, \dots, \beta_k} + \cdots + \det M_{\beta_1,\beta_2, \dots, \beta_k+1}. \]

Beginning with $\det M = \det M_{0, \dots, 0}$, its first derivative is $\det M_{0, \dots, 0, 1}$, which is divisible by $(y-x)$; all other terms are $0$. The second derivative is then $\det M_{0, \dots, 1,1} + \det M_{0, \dots, 0, 2}$, which for $k \ge 2$ is still divisible by $(y-x)$. Inductively, the $j$th derivative of $\det M_{0, \dots, 0}$ is 
\[\sum_{\substack{\text{increasing paths $P$} \\ \text{of length $j$}}} \det M_{P(j)},\]
where the sum is over all paths $P$ beginning at $(0, \dots, 0) \in \mathbb Z^k$, of length $j$, where each step has nondecreasing entries. If $\beta_1\leq k-1$, then $(y-x)$ divides $\det M_{P(j)}$. Otherwise, $\beta_1\geq k$, and either $\det M_{P(j)}=0$ or $\beta_i\geq k$, leading to $\sum_{i=1}^k \beta_k\geq k^2$, which is a contradiction if we are taking the $j$th derivative with 
$j \le k^2-1$. Thus the first $k^2-1$ derivatives are all divisible by $(y-x)$, so $\det M$ is a  multiple of $(y-x)^{k^2}$, as desired. 
\end{proof}

\section{The symmetric function theory context in the symplectic case} \label{sec:symmetric-symplectic}
In this section, we describe a model for $I_{d_k,2}^S(n;N)$ using symmetric function theory, and show that the degree of $I_{d_k,2}^S(c2N;N)$ as a function of $N$ is $2k^2+k-2$. We follow closely Section 4.3 of \cite{KR3}.

If $n$ is a positive integer,  a \emph{partition} $\lambda$ of $n$ is a sequence of positive integers $\lambda_1\geq \lambda_2\geq \cdots \geq \lambda_r$ such that $n=|\lambda|:=\lambda_1+\cdots+\lambda_r$. The \emph{length} of $\lambda$ is defined by $\ell(\lambda)=r$. A \emph{Ferrer diagram} is a collection of ``cells'' of length $\lambda_i$ arranged in left-justified rows, with $\lambda_i$ cells in the $i$th row. A Ferrer diagram is a \emph{semi-standard Young tableau} (SSYT) when the cells are labeled by integers $(T_{i,j})_{\substack{1\leq i \leq \ell(\lambda)\\ 1 \leq j \leq \lambda_i}}$  in such a way that the rows are non-decreasing and the columns are increasing, that is $T_{i,j}\leq T_{i,j+1}$ and $T_{i,j}<T_{i+1,j}$ (see Figure \ref{fig:ssyt}). An SSYT $T$ is of \emph{shape} $\lambda$ if the Ferrer diagram of the tableau is a Ferrer diagram for $\lambda$. 
\begin{figure}[h] 
\begin{ytableau}
       1 & 1 & 2 & 3 & 5& 5 \\
  2 & 2& 5 &5\\
  4 &5 \\
  5&8
  \end{ytableau}
  \caption{ \label{fig:ssyt}Semi-standard Young tableau of shape $(6,4,2,2)$}
\end{figure}


We say that an SSYT $T$ has \emph{type} $(a_1,a_2,\dots)$ if $T$ has $a_i=a_i(T)$ cells equal to $i$, and use the notation 
\[x^T=x_1^{a_1(T)}x_2^{a_2(T)}\cdots.\]
In the example of Figure \ref{fig:ssyt}, $x^T=x_1^2x_2^3x_3x_4x_5^6x_8$.

For a partition $\lambda$, the \emph{Schur function} in the variables $x_1,\dots, x_r$ indexed by $\lambda$ is an $r$-variable polynomial given by 
\[s_\lambda(x_1,\dots,x_r)=\sum_T x_1^{a_1(T)}\cdots x_r^{a_r(T)},\]
where the sum is taken over all the SSYTs $T$ of shape $\lambda$ whose entries belong to the set $\{1,\dots, r\}$ (so that $a_i(T)=0$ for $i>r$).

Proposition 11 and equation (88) of \cite{Bump-Gamburd} imply that
\begin{equation}\label{eq:BG}\int_{\mathrm{Sp}(2N)} \prod_{j=1}^r \det(1+Ux_j) \mathrm{d}U =\sum_{\substack{\lambda_1\leq 2N\\\lambda \,\text{even}}}s_\lambda(x_1,\dots,x_r),
\end{equation}
where the sum is taken over even partitions, that is, partitions where $\lambda_i$ is even for all $i$. 

We are interested in understanding $I_{d_k,2}^S(n;N)$ given by  \eqref{symplecticintegral}.  We therefore focus on the case $r=2k$, and $x_1=\cdots =x_k=x$,  $x_{k+1}=\cdots =x_{2k}=y$ in \eqref{eq:BG}. Write 
\begin{equation}\label{eq:gen}
\int_{\mathrm{Sp}(2N)} \det(1+Ux)^k\det(1+Uy)^k \mathrm{d}U=\sum_{m,n=0}^{2Nk}x^m y^nJ_{d_k,2}^S(m,n;N),
\end{equation}
so that we are interested in the diagonal terms 
\begin{equation*}
I_{d_k,2}^S(n;N):=J_{d_k,2}^S(n,n;N).
\end{equation*}
Thus
\[I_{d_k,2}^S(n;N)=\left[\sum_{\substack{\lambda_1\leq 2N\\\lambda \,\text{even}}}s_\lambda(\underbrace{x,\dots,x}_k,\underbrace{y, \dots, y}_k) \right]_{x^ny^n}\]
Then $I_{d_k,2}^S(n;N)$ is the number of SSYTs $T$ with even shape such that $\lambda_1\leq 2N$ and 
\[a_1+\cdots +a_k=n, \qquad a_{k+1}+\cdots+a_{2k}=n.\]

Following ideas from \cite{KR3}, parametrize the $T$ by letting $y_r^{(s)}=y_r^{(s)}(T)$ be the rightmost position of the entry $s$ in the row $r$. If $s$ does not occur in row $r$, inductively define $y_r^{(s)}=y_r^{(s-1)}$, with $y_r^{(1)}=0$ if the entry $1$ does not occur in row $r$. Thus, each $T$ corresponds to an array of shape
\[\begin{bmatrix}
   y_1^{(1)} &  y_1^{(2)} &  \dots &  y_1^{(2k)}\\
    0&  y_2^{(2)} &  \dots &  y_2^{(2k)}\\
   \vdots & \ddots & \ddots & \vdots  \\
   0 & 0 & \cdots &  y_{2k}^{(2k)}
  \end{bmatrix},\]
which satisfies the following properties:
\begin{enumerate}
 \item $y_r^{(s)} \in \Z\cap [0,2N]$, 
 \item $y_{r+1}^{(s+1)}\leq y_r^{(s)}$,
 \item $y_r^{(s)}\leq y_r^{(s+1)}$,
 \item $y_1^{(k)}+\cdots+y_k^{(k)}=n$,
 \item $y_1^{(2k)}+\cdots+y_{2k}^{(2k)}=2n$, and 
  \item $y_1^{(2k)},\dots, y_{2k}^{(2k)}$ are even.
\end{enumerate}
The last condition codifies the fact that $\lambda$ is even and will play a special role in estimating the number of admissible SSYTs.

In the example of Figure \ref{fig:ssyt}, taking $k=4$ and $2N\geq 6$, the corresponding matrix is 
\[\begin{bmatrix}
   2 &  3 &  4 &  4 & 6 & 6 & 6 & 6\\
0 &  2 &  2 &  2 & 4 & 4 & 4 & 4\\
0 &  0 &  0 &  1 & 2 & 2 & 2 & 2\\
0 &  0 &  0 &  0& 1 & 1 & 1 & 2\\
0 &  0 &  0 &  0& 0 & 0 & 0 & 0\\
0 &  0 &  0 &  0& 0 & 0 & 0 & 0\\
0 &  0 &  0 &  0& 0 & 0 & 0 & 0\\
0 &  0 &  0 &  0& 0 & 0 & 0 & 0\\
   \end{bmatrix},\]
where $n=7$ and $\lambda_1=6\leq 2N$.  
 
Let $c =\frac{n}{2N}$ and let $U^S=\{(i,j): 1\leq i \leq j \leq 2k, (i,j)\not = (1,k), (1,2k)\}$. Let $V^S_c$ be the convex region contained in $\R^{2k^2+k-2}=\{(u_i^{(j)})_{(i,j)\in U^S} : u_i^{(j)} \in \R\}$ defined by the following inequalities:
\begin{enumerate}
 \item $0\leq u_i^{(j)} \leq 1$ for $(i,j) \in U^S$,
 \item for $u_1^{(k)}:=c-(u_2^{(k)}+\cdots+u_k^{(k)})$, we have $0\leq u_1^{(k)}\leq 1$,
  \item for $u_1^{(2k)}:=2c-(u_2^{(2k)}+\cdots+u_{2k}^{(2k)})$, we have $0\leq u_1^{(2k)}\leq 1$,
   \item $u_{r+1}^{(s+1)}\leq u_r^{(s)}$, and
 \item $u_r^{(s)}\leq u_r^{(s+1)}$. 
\end{enumerate}
Finally, let $A_k^S$ be the set defined by the last two conditions above. 

The region $V^S_c$ is convex because it is the intersection of half-planes. Note that for $c \in [0,k]$, $V^S_c$ is contained in $[0,1]^{2k^2+k-2}$ and therefore contained in a closed ball of diameter $\sqrt{2k^2+k-2}$. Let $\Z^{2k^2-k-1, 2k-1}_\text{e}\cong \Z^{2k^2-k-1} \times (2\Z)^{2k-1}$, where the even coordinates correspond to the elements $(i, 2k)\in U^S$ for $i=2,\dots, 2k$. Then
\begin{equation}\label{eq:Ivol}
I_{d_k,2}^S(n;N)=\# \left(\Z^{2k^2-k-1, 2k-1}_\text{e} \cap (2N\cdot V^S_c) \right),
\end{equation}
where $2N\cdot V^S_c=\{ 2Nx: x\in V_c\}$ is the dilate of $V^S_c$ by a factor of $2N$. 

We now use the well-known principle that the count of lattice points in a region can be approximated by the volume of the region. Let $\Lambda \subset \R^\ell$ be a lattice. The shape of $\Lambda$  is roughly described by the successive minima $\mu_1\leq \cdots \leq \mu_\ell$ of $\Lambda$. 
\begin{thm}\cite[Lemma 1]{Schmidt} \label{lem:schmidt}
If $S\subset \R^\ell$ is a convex region contained in a closed ball of radius $\rho$, and if $\mu_{\ell-1}\leq \rho$, then 
\[\# (\Lambda \cap S)=\frac{\mathrm{vol}_\ell(S)}{\det\Lambda} +O\left( \frac{\mu_\ell \rho^{\ell-1}}{\det \Lambda}\right).\]
\end{thm}
Applying this to $\Lambda=\Z^{2k^2-k-1, 2k-1}_\text{e}$, we get 
\[I_{d_k,2}^S(n;N)=\frac{\mathrm{vol}_{2k^2+k-2}(2N\cdot V^S_c)}{2^{2k-1}} +O_k\left( N^{2k^2+k-3}\right).\]
The volume of the region is given by 
\begin{equation}\label{eq:volume}\begin{aligned}&\mathrm{vol}_{2k^2+k-2}(2N\cdot V^S_c)\\&=(2N)^{2k^2+k-2} \int_{[0,1]^{2k^2+k}} \delta_c(u_1^{(k)}+\cdots+u_k^{(k)}) \delta_{2c}(u_1^{(2k)}+\cdots+u_{2k}^{(2k)})\mathds{1}_{A_k^S}(u) d^{2k^2+k} u.\end{aligned}\end{equation}

To confirm that $I^S_{d_k,2}(n;N)$ is a polynomial, we rely on  Ehrhart theory \cite{Ehrhart} (see also \cite{Breuer} for details).

\begin{thm}\cite{Ehrhart} \label{thm:Ehrhart} Let $E$ be a convex lattice polytope in $\R^\ell$ (that is, a polytope whose vertices have all integral coordinates). 
Then there is a polynomial $P$ of degree $\ell$ such that for all positive integers $m$, 
\[\# \left( \Z^\ell \cap (m\cdot E)\right)=P(m).\]
 \end{thm}

We are now ready to prove Theorem \ref{thm:degreeEhrharts}.
\begin{cor} \label{cor:degreeEhrharts}
 Let $c=\frac{a}{b}$ be a fixed rational number and $k$ be a fixed integer. If $2N$ is a multiple of $b$, then $I_{d_k,2}^S(c2N;N)= P_{c,k}(N)$, where $P_{c,k}$ is a polynomial of degree ${2k^2+k-2}$.
\end{cor}

\begin{proof}
 Writing $2N= bm$ and applying \eqref{eq:Ivol},
 \[I_{d_k,2}^S(c2N;N)=\# \left(\Z^{2k^2-k-1, 2k-1}_\text{e} \cap (m\cdot\left[b \cdot V^S_c\right]) \right).
\]
Since $b\cdot V^S_c$ is a convex lattice polytope in $\R^{(2k^2-k-1)+ (2k-1)}$, Theorem \ref{thm:Ehrhart} implies that $I_{d_k,2}^S(c2N;N)$ is a polynomial of degree ${2k^2+k-2}$ in $m$, and therefore in $N$. 
\end{proof}

For what follows we will need the following result of \cite{KR3}.
\begin{lem}{\cite[Lemma 4.6]{KR3}} \label{lem:KR3}
 For $\beta \in \R^{k+1}$ satisfying $\beta_1\leq \beta_2\leq \cdots \beta_{k+1}$, define
 \[I(\beta)=\{\alpha \in \R^{k}\, :\,\beta_1\leq \alpha_1\leq \beta_2\leq \alpha_2\leq \cdots \alpha_k\leq \beta_{k+1}\}.\]
 Then 
 \[\int_{\alpha \in I(\beta)} \Delta(\alpha_k,\alpha_{k-1},\dots, \alpha_1)d^k\alpha = \frac{1}{k!}\Delta(\beta_{k+1},\dots,\beta_1),
 \]
where $\Delta$ denotes the Vandermonde determinant.  

\end{lem}

\begin{prop}\label{eq:intsymp} The coefficient of $(2N)^{2k^2+k-2}$ in the asymptotics for $I_{d_k,2}^S(c(2N),N)$ is given by 
 \begin{align} \label{defi:gammaS}
 \gamma_{d_k,2}^S(c)=&\frac{2^{-2k+1}}{G(1+k)}\int_{[0,1]^{\frac{3}{2}k^2+\frac{3}{2}k}} \delta_c(u_1^{(k)}+\cdots+u_k^{(k)}) \delta_{2c}(u_1^{(2k)}+\cdots+u_{2k}^{(2k)})\nonumber\\
&\times \mathds{1}   \begin{bmatrix}
    &   & u_1^{(k)} & \leq &   \dots &\leq & u_1^{(2k)}\\
    &   \iddots &  &   & & \iddots& \\
     u_k^{(k)} & \leq & \dots & \leq & u_k^{(2k)} &\\
   \vdots & &  \iddots& &  &&   \\
   u_{2k}^{(2k)} && & &&&
  \end{bmatrix}
\Delta(u_1^{(k)},u_2^{(k)},\dots,u_k^{(k)}) d^{\frac{3}{2}k^2+\frac{3}{2}k} u,
\end{align}
where $\mathds{1}_X$ denotes the characteristic function of the set $X$. 
\end{prop}

\begin{proof}
From the volume formula \eqref{eq:volume}, we have
\begin{align}\label{eq:intgamma}
 \gamma_{d_k,2}^S(c)=&2^{-2k+1}\int_{[0,1]^{2k^2+k}} \delta_c(u_1^{(k)}+\cdots+u_k^{(k)}) \delta_{2c}(u_1^{(2k)}+\cdots+u_{2k}^{(2k)})\nonumber\\
&\times \mathds{1} \begin{bmatrix}
   u_1^{(1)} & \leq & u_1^{(2)} & \leq & \dots & \leq &   u_1^{(2k-1)} &\leq & u_1^{(2k)}\\
   \vertgeq  & &   \vertgeq   & & & &   \vertgeq  & & \\
     u_2^{(2)} & \leq & u_2^{(3)} & \leq & \dots & \leq & u_2^{(2k)} &&\\
   \vdots & \vdots & \vdots & \vdots &\iddots  & && &   \\
   u_{2k}^{(2k)} && & &&&&&
  \end{bmatrix}
 d^{2k^2+k} u.
\end{align}
Integrating with respect to $u_1^{(1)}$ in \eqref{eq:intgamma} and applying Lemma \ref{lem:KR3}, 
\begin{align*}
 \gamma_{d_k,2}^S(c)=&2^{-2k+1}\int_{[0,1]^{2k^2+k-1}} \delta_c(u_1^{(k)}+\cdots+u_k^{(k)}) \delta_{2c}(u_1^{(2k)}+\cdots+u_{2k}^{(2k)})\\
&\times \mathds{1}  \begin{bmatrix}
    & & u_1^{(2)} & \leq & \dots & \leq &   u_1^{(2k-1)} &\leq & u_1^{(2k)}\\
    & &   \vertgeq   & & & &   \vertgeq  & & \\
     u_2^{(2)} & \leq & u_2^{(3)} & \leq & \dots & \leq & u_2^{(2k)} &&\\
   \vdots & \vdots & \vdots & \vdots &\iddots  &&& &   \\
   u_{2k}^{(2k)} && & &&&&&
  \end{bmatrix}
 \frac{\Delta(u_1^{(2)}, u_2^{(2)})}{1!}d^{2k^2+k-1} u.
\end{align*}
We proceed inductively and get
\begin{align*}
 \gamma_{d_k,2}^S(c)=&2^{-2k+1}\int_{[0,1]^{\frac{3}{2}k^2+\frac{3}{2}k}} \delta_c(u_1^{(k)}+\cdots+u_k^{(k)}) \delta_{2c}(u_1^{(2k)}+\cdots+u_{2k}^{(2k)})\\
&\times \mathds{1}   \begin{bmatrix}
    &   & u_1^{(k)} & \leq &   \dots &\leq & u_1^{(2k)}\\
    &   \iddots &  &   & & \iddots& \\
     u_k^{(k)} & \leq & \dots & \leq & u_k^{(2k)} &\\
   \vdots & &\iddots  & &  &&   \\
   u_{2k}^{(2k)} && & &&&
  \end{bmatrix}
\frac{\Delta(u_1^{(k)},u_2^{(k)},\dots,u_k^{(k)})}{1!2!\cdots (k-1)!} d^{\frac{3}{2}k^2+\frac{3}{2}k} u.
\end{align*}

\end{proof}

\section{The determinant expression in the symplectic case}\label{sec:determinant-symplectic}

Consider in full generality 
\[I_{d_k,\ell}^S(n;N):=\int_{\mathrm{Sp}(2N)} \Big|\sum_{\substack{j_1+\cdots+j_k=n\\0\leq j_1,\dots,j_k \leq 2N}}\mathrm{Sc}_{j_1}(U)\cdots \mathrm{Sc}_{j_k}(U)\Big|^\ell \mathrm{d}U,\]
and write
\[I_{d_k,\ell}^S(n;N) \sim \gamma_{d_k,\ell}^S(c) (2N)^m,\]
where $m$ is the appropriate exponent of $N$ such that the above asymptotic formula exists. 

The case of $I_{d_k,1}^S(n;N)$ and $\gamma_{d_k,1}^S(c)$ was studied in detail by Medjedovic  \cite{Andy}. However, for the purposes of understanding the asymptotics given in \eqref{eq:intsym}, we need to consider $I_{d_k,2}^S(n;N)$ and $\gamma_{d_k,2}^S(c)$ instead. The goal of this section is to further develop the results of Medjedovic  to adapt them to our case. Following \cite{Andy}, the D\'esarm\'enien--Stembridge--Proctor formula \cite{BenderKnuth,BeteaWheeler,Proctor} gives 
\begin{align}\label{eq:DSP}
\sum_{\substack{\lambda_1\leq 2N\\\lambda \,\text{even}}}s_\lambda(x_1,\dots,x_r)=\frac{1}{\Delta(x_1,\dots,x_r)}\prod_{i=1}^r \frac{1}{1-x_i^2} \prod_{i<j} \frac{1}{1-x_ix_j} \det_{1\leq i,j \leq r} 
\left[x_i^{2N+2r+1-j}-x_i^{j-1}\right].
\end{align}
(See also equation (53) in \cite{Bump-Gamburd}.)

Combining equations \eqref{eq:BG}, \eqref{eq:gen}, and \eqref{eq:DSP}, we have
\begin{align*}
& \sum_{m,n=0}^{2Nk}x^m y^nJ_{d_k,2}^S(m,n;N)\\=&\frac{1}{(1-x^2)^{\binom{k+1}{2}}(1-y^2)^{\binom{k+1}{2}}(1-xy)^{k^2}}\left. \frac{\det_{1\leq i,j \leq 2k} \left[x_i^{2N+4k+1-j}-x_i^{j-1}\right]}{\Delta(x_1,\dots,x_{2k})}\right|_{(x,\dots,x,y,\dots,y)}.
\end{align*}
Applying Lemma \ref{lem:superAndy}, we obtain
\begin{align} \label{eq:symplecticdet}
& \sum_{m,n=0}^{2Nk}x^m y^nJ_{d_k,2}^S(m,n;N)\nonumber\\
 =&\frac{1}{(1-x^2)^{\binom{k+1}{2}}(1-y^2)^{\binom{k+1}{2}}(1-xy)^{k^2}(y-x)^{k^2}}
 \det_{\substack{1\leq i_1,i_2 \leq k\\1\leq j \leq 2k}} \left[\begin{array}{c}
 \binom{2N+4k+1-j}{i_1-1} x^{2N+4k+2-j-i_1}-\binom{j-1}{i_1-1} x^{j-i_1}\\ \\
 \binom{2N+4k+1-j}{i_2-1} y^{2N+4k+2-j-i_2}- \binom{j-1}{i_2-1} y^{j-i_2}\end{array}\right].
  \end{align}
It is not clear a priori that the right-hand side of \eqref{eq:symplecticdet} leads to a polynomial or even a Taylor series, due to the term $(y-x)^{k^2}$ in the denominator, but  this follows from Proposition \ref{prop:superVivian}.

There is a symmetry for the variable $n$ in $I_{d_k,2}^S(n;N)$ as illustrated next. 
 \begin{lem} \label{lem:funceq} For $0\leq n \leq 2Nk$ the following functional equation holds:
  \[I_{d_k,2}^S(n;N)=I_{d_k,2}^S(2Nk-n;N).\]
 \end{lem}
 \begin{proof}
  The functional equation of the characteristic polynomial of a symplectic matrix is given by 
  \[\det(I+Ux)=x^{2N} \det(I+Ux^{-1}).\]
  Then, from equation \eqref{eq:gen} we have 
 \[\sum_{m,n=0}^{2Nk}x^m y^nJ_{d_k,2}^S(m,n;N)=x^{2Nk}y^{2Nk} \sum_{m,n=0}^{2Nk}x^{-m} y^{-n}J_{d_k,2}^S(m,n;N).\]
 By making the changes of variables $m\rightarrow 2Nk-m$, $n\rightarrow 2Nk-n$, we immediately obtain 
 \[J_{d_k,2}^S(m,n;N)=J_{d_k,2}^S(2Nk-m,2Nk-n;N),\]
 and we deduce the result from setting $m=n$. 
 \end{proof}

 Now we focus on computing $I_{d_k,2}^S(n;N)$ for the case $n\leq N+\frac{1+k}{2}$.
 
Before stating the result we make the following definition. A \emph{quasi-polynomial} $P(m)$ of period $r$ is a function on integer numbers for which there exist polynomials $P_0(m),\dots,P_{r-1}(m)$ such that \[P(m)=\begin{cases}        P_0(m) & m\equiv 0 \pmod{r},\\P_1(m) & m\equiv 1 \pmod{r},\\ \vdots &\\P_{r-1}(m) & m\equiv r-1 \pmod{r}.\end{cases}\]
 
 \begin{prop} \label{prop:nless2N}
  For $n\leq N+\frac{1+k}{2}$, we have 
  \[I_{d_k,2}^S(n;N)=\frac{1}{G(1+k)}\sum_{\substack{\ell=0\\\ell\equiv n \pmod{2}}}^n\binom{\frac{n-\ell}{2}+\binom{k+1}{2}-1}{\binom{k+1}{2}-1}^2\binom{\ell+k^2-1}{k^2-1}.\]
 Moreover, $I_{d_k,2}^S(n;N)$ is a quasi-polynomial in $n$ of period 2 and degree $2k^2+k-2$ (provided that $n\leq N+\frac{1+k}{2}$).
 \end{prop}

\begin{proof} From the right-hand side of equation \eqref{eq:symplecticdet}, the contribution to $I_{d_k,2}^S(n;N)$, that is, the coefficient of $x^ny^n$,  for $n\leq N+\frac{1+k}{2}$, comes from the determinant that takes exclusively terms of the form $\binom{j-1}{i-1}x^{j-i}$ and similarly with $y$. In other words, $I_{d_k,2}^S(n;N)$ comes exclusively from the diagonal terms of 
\begin{align*}
 \frac{1}{(1-x^2)^{\binom{k+1}{2}}(1-y^2)^{\binom{k+1}{2}}(1-xy)^{k^2}(y-x)^{k^2}}
 \det_{\substack{1\leq i_1,i_2 \leq k\\1\leq j \leq 2k}} \left[\begin{array}{c}\binom{j-1}{i_1-1} x^{j-i_1}\\ \\
 \binom{j-1}{i_2-1} y^{j-i_2}\end{array}\right].
\end{align*}
 Lemma \ref{lem:y-xk2} implies that we must consider the diagonal terms in 
 \begin{align*}
& \frac{1}{G(1+k)(1-x^2)^{\binom{k+1}{2}}(1-y^2)^{\binom{k+1}{2}}(1-xy)^{k^2}} \\=& \frac{1}{G(1+k)} \sum_{\ell_1=0}^\infty \binom{\ell_1+\binom{k+1}{2}-1}{\binom{k+1}{2}-1}x^{2\ell_1} \sum_{\ell_2=0}^\infty \binom{\ell_2+\binom{k+1}{2}-1}{\binom{k+1}{2}-1}y^{2\ell_2} \sum_{m=0}^\infty \binom{m+k^2-1}{k^2-1}(xy)^m.
\end{align*}
We obtain
\begin{align*}
 I_{d_k,2}^S(n;N)=\frac{1}{G(1+k)}\sum_{\ell=0}^{\lfloor\frac{n}{2}\rfloor}\binom{\ell+\binom{k+1}{2}-1}{\binom{k+1}{2}-1}^2\binom{n-2\ell +k^2-1}{k^2-1}.
\end{align*}
This is a quasi-polynomial in $n$, since the upper bound in the sum depends on the parity of $n$. Each term in the above sum has total degree $2\left(\binom{k+1}{2}-1\right)+k^2-1=2k^2+k-3$ in the variable $\ell$. By summing over $\ell$, each term leads to a polynomial of degree $2k^2+k-2$ on $n$. Moreover, the coefficients are all positive, which guarantees that there is no cancelation that could lower the final degree. 
\end{proof}


\subsection{Some low degree cases}
Here we specialize Proposition \ref{prop:nless2N}  for low values of $k$. For $k=1$ we have 
\[I_{d_1,2}^S(n;N)=\left \lfloor \frac{n+2}{2}\right \rfloor=\frac{2n+3}{4}+\frac{(-1)^n}{4}.\]
Combining with Lemma \ref{lem:funceq}, this recovers \cite[Corollary 5.12]{KuperbergLalin}:
 If $k = 1$, then 
 \[I_{d_1,2}^S(n,N):=\int_{\mathrm{Sp}(2N)}\left|\mathrm{Sc}_n(U)\right|^2 \mathrm{d}U =\begin{cases}
             \left \lfloor \frac{n+2}{2} \right\rfloor & 0\leq n\leq N,\\
              \left \lfloor \frac{2N-n+2}{2} \right\rfloor  & N+1\leq n \leq 2N.
             \end{cases}\]
As $N \to \infty$, $I^S_{d_1,2}(n,N)$ is asymptotic to 
\[I^S_{d_1,2}(n,N) \sim \gamma_{d_1,2}^S(c) 2N,\]
where $c = \frac{n}{2N}$ and
   \[\gamma_{d_1,2}^S(c)=\begin{cases}
             \frac c2  & 0\leq c\leq \frac{1}{2},\\
             \frac{1-c}2  & \frac{1}{2}\leq c \leq 1.
             \end{cases}\]

For $k=2$ we have 
\begin{align*}
I_{d_2,2}^S(n;N)=&\sum_{\substack{\ell=0\\\ell\equiv n \pmod{2}}}^n\binom{\frac{n-\ell}{2}+2}{2}^2 \binom{\ell+3}{3}\\
=&\frac{1}{1290240}(6n^8+240n^7+4088n^6+38640n^5+221354n^4+787080n^3\\&+1698572n^2+2031720n+1018395)
+\frac{(-1)^n}{4096}(2n^4+40n^3+284n^2+840n+863).
\end{align*}
This leads to
\[I_{d_2,2}^S(n;N) \sim \gamma_{d_2,2}^S(c)(2N)^8,\]
where 
  \[\gamma_{d_2,2}^S(c)=\begin{cases}
             \frac{c^8}{215040}  & 0\leq c\leq \frac{1}{2},\\
             \frac{(2-c)^8}{215040}  & \frac{3}{2}\leq c \leq 2.
             \end{cases}\]

\section{A complex analysis approach in the symplectic case} \label{sec:complex-symplectic}

In this section we further study the asymptotic coefficient $\gamma_{d_k,2}^S(c)$ by studying the integral \eqref{eq:gen} with tools from complex analysis. 

First, recall equation (3.36) from \cite{CFKRS-autocorrelation}:
 \begin{align*}
\int_{\mathrm{Sp}(2N)} \prod_{j=1}^r\det(I-Ue^{-\alpha_j}) \mathrm{d}U =& \frac{(-1)^{r(r-1)/2}2^r}{(2\pi i)^r r!} e^{-N\sum_{j=1}^r \alpha_j} \oint \cdots \oint \prod_{1\leq \ell \leq m \leq r} (1-e^{-z_m-z_\ell})^{-1}\\
  &\times \frac{\Delta(z_1^2,\dots,z_r^2)^2 \prod_{j=1}^r z_je^{N\sum_{j=1}^r z_j} \mathrm{d}z_1\cdots \mathrm{d}z_r}{\prod_{i=1}^r \prod_{j=1}^r (z_j-\alpha_i)(z_j+\alpha_i)}.
 \end{align*}
By \eqref{eq:gen},  we are interested in the diagonal coefficients of the polynomial
  \begin{align*}
 P_k(x,y):= &\int_{\mathrm{Sp}(2N)} \det(I-Ux)^k \det(I-Uy)^k\mathrm{d}U\\
 =& \frac{(-1)^{k}2^{2k}}{(2\pi i)^{2k} (2k)!} x^{Nk} y^{Nk} \oint \cdots \oint \prod_{1\leq \ell \leq m \leq 2k} (1-e^{-z_m-z_\ell})^{-1}\\
  &\times \frac{\Delta(z_1^2,\dots,z_{2k}^2)^2 \prod_{j=1}^{2k} z_j e^{N\sum_{j=1}^{2k} z_j} \mathrm{d}z_1\cdots \mathrm{d}z_{2k}}{\prod_{j=1}^{2k} (z_j-\log(x))^k(z_j+\log(x))^k(z_j-\log(y))^k(z_j+\log(y))^k}.
 \end{align*}
In order to compute the integral $P_k(x,y)$, first shrink the contour of the integrals into small circles centered at $\pm \log(x)$  and $\pm \log(y)$.  More precisely, let  $\epsilon_j=\log x $ or $\log y$ and $\delta_j=\pm 1$ for $j=1,\dots, 2k$. Consider
 \begin{align}\label{eq:pkdelta}
 P_k(x,y,\delta_1\epsilon_1,\dots \delta_{2k}\epsilon_{2k}):=& \oint \cdots \oint \prod_{1\leq \ell \leq m \leq 2k} (1-e^{-z_m-z_\ell})^{-1}\nonumber \\
  &\times \frac{\Delta(z_1^2,\dots,z_{2k}^2)^2 \prod_{j=1}^{2k} z_je^{N\sum_{j=1}^{2k} z_j} \mathrm{d}z_1\cdots \mathrm{d}z_{2k}}{\prod_{j=1}^{2k} (z_j-\log(x))^k(z_j+\log(x))^k(z_j-\log(y))^k(z_j+\log(y))^k},
 \end{align}
 where the $j$th integral is around a small circle centered at $\delta_j \epsilon_j$ for $j=1,\dots, 2k$.  This gives $4^{2k}$ multiple integrals, and 
\begin{equation*}
 P_k(x,y)= \frac{(-1)^{k}2^{2k}}{(2\pi i)^{2k} (2k)!} x^{Nk} y^{Nk} \sum_{\delta_j, \epsilon_j} P_k(x,y,\delta_1\epsilon_1,\dots \delta_{2k}\epsilon_{2k}),
 \end{equation*}
 where the sum is taken over all the possible $4^{2k}$ choices for $\delta_j$ and $\epsilon_j$. 
 
The following lemma shows that the only nonzero terms in the above sum are those where $\epsilon_j = \log x$ for exactly half of the $j$.
 
\begin{lem}\label{lem:4.11}(Equivalent of \cite[Lemma 4.11]{KR3}) Let $P_k(x,y,\delta_1\epsilon_1,\dots \delta_{2k}\epsilon_{2k})$  be defined by \eqref{eq:pkdelta}. Then 
  $P_k(x,y,\delta_1\epsilon_1,\dots \delta_{2k}\epsilon_{2k})=0$  unless exactly half of the $\epsilon_j$'s equal $\log x$. 
\end{lem}
\begin{proof}
 Assume that $k+1$ of the $\epsilon_j$'s are $\log x$, and without loss of generality that $\epsilon_1=\cdots=\epsilon_{k+1}=\log x$. Let
 \begin{align*}
  G(z_1,\dots,z_{k+1})=&\prod_{1\leq \ell \leq m \leq 2k} (1-e^{-z_m-z_\ell})^{-1}\\
  &\times \frac{\Delta(z_1^2,\dots,z_{2k}^2) \prod_{j=1}^{2k} z_je^{N\sum_{j=1}^{2k} z_j} \prod_{j=1}^{k+1}(z_j-\delta_j\log(x))^k }{\prod_{j=1}^{2k} (z_j-\log(x))^k(z_j+\log(x))^k(z_j-\log(y))^k(z_j+\log(y))^k}.
 \end{align*}
Poles arising from $z_i=-z_j$ cancel with the Vandermonde determinant in the numerator, so $G(z_1,\dots,z_{k+1})$ is analytic around $(\delta_1\log x, \dots, \delta_{k+1}\log x)$. Now use the residue theorem to compute the original integral in \eqref{eq:pkdelta}, whose integrand is 
\[  \frac{G(z_1,\dots,z_{k+1})\Delta(z_1^2,\dots,z_{2k}^2) }{\prod_{j=1}^{k+1}(z_j-\delta_j\log(x))^k }.\]
To compute the residue we need to compute the coefficient of $ \prod_{j=1}^{k+1}(z_j-\delta_j\log(x))^{k-1}$ in the numerator. By properties of the Vandermonde, 
\begin{align*}
\Delta(z_1^2,\dots,z_{2k}^2)=&\Delta(z_1^2-(\log x)^2,\dots,z_{2k}^2-(\log x)^2)\\
=& \sum_{\sigma \in \mathbb{S}_{2k}} \mathrm{sgn}(\sigma) \prod_{i=1}^{2k} \left( z_i^2-(\delta_j\log x)^2\right)^{\sigma(i)-1},
\end{align*}
and it is impossible to have $\sigma(i)-1\leq k-1$ for all $i=1,\dots,k+1$ at the same time, so the desired coefficient is $0$.
\end{proof}
 
By Lemma \ref{lem:4.11} and symmetry, we can assume that $\epsilon_1=\dots=\epsilon_k=\log x$ and $\epsilon_{k+1}=\dots=\epsilon_{2k}=\log y$. Fix those values and write 
 \begin{equation}\label{eq:Pk}
 P_k(x,y)= \frac{(-1)^{k}2^{2k}}{(2\pi i)^{2k} (2k)!} x^{Nk} y^{Nk} \binom{2k}{k}  \sum_{\delta_j} P_k(x,y,\delta_1\epsilon_1,\dots \delta_{2k}\epsilon_{2k}).
 \end{equation}
 
We now compute $P_k(x,y,\delta_1\epsilon_1,\dots \delta_{2k}\epsilon_{2k})$. First make the change of variables $z_j=\delta_j\epsilon_j+\frac{v_j}{N}$. Then the integrand of $P_k(x,y,\delta_1\epsilon_1,\dots \delta_{2k}\epsilon_{2k})$ becomes (up to terms of order $1/N$ smaller)
 \begin{align*}
 &\prod_{1\leq \ell \leq m \leq k} (1- x^{-\delta_m-\delta_\ell}e^{\frac{-v_m-v_\ell}{N}})^{-1}
 \prod_{1\leq\ell  \leq k<m \leq 2k}(1- x^{-\delta_\ell} y^{-\delta_m} e^{\frac{-v_m-v_\ell}{N}})^{-1}
 \prod_{k< \ell \leq m \leq 2k}(1- y^{-\delta_m-\delta_\ell}e^{\frac{-v_m-v_\ell}{N}})^{-1}\\
  & \times (2\log x/N)^{k(k-1)} \prod_{1\leq \ell < m \leq k} (\delta_\ell v_\ell -\delta_m v_m)^2 (\log^2y -\log^2 x)^{2k^2}  (2\log y/N)^{k(k-1)} \prod_{k< \ell < m \leq 2k} (\delta_\ell v_\ell -\delta_m v_m)^2\\
  &\times (\prod_{j=1}^{2k} \delta_j) (\log x)^k (\log y)^kx^{N\sum_{i=1}^k \delta_i} y^{N\sum_{i=k+1}^{2k} \delta_i} e^{\sum_{j=1}^{2k} v_j} \\
  &\times \frac{1}{\prod_{j=1}^{k} \frac{v_j^k}{N^k}\left( 2\delta_j \log x +\frac{v_j}{N}\right)^k
  \left(\log^2x -\log^2 y +2\delta_j \log x \frac{v_j}{N}+\frac{v_j^2}{N^2}\right)^k}\\
  &\times \frac{1}{
  \prod_{j=k+1}^{2k} \frac{v_j^k}{N^k}\left( 2\delta_j \log y +\frac{v_j}{N}\right)^k
  \left(\log^2y -\log^2 x +2\delta_j \log y \frac{v_j}{N}+\frac{v_j^2}{N^2}\right)^k}\times \frac{\mathrm{d}v_1\cdots \mathrm{d}v_{2k}}{N^{2k}}.
 \end{align*}
 After some further simplification, the above becomes, as $N \rightarrow \infty$,
\begin{align}\label{eq:page17}
\sim &2^{2k(k-1)} (\log^2y -\log^2 x)^{2k^2}  (\log x)^{k^2} (\log y)^{k^2} (\prod_{j=1}^{2k} \delta_j)x^{N\sum_{i=1}^k \delta_i} y^{N\sum_{i=k+1}^{2k} \delta_i}\nonumber \\&\times \prod_{1\leq \ell \leq m \leq k} (1- x^{-\delta_m-\delta_\ell}e^{\frac{-v_m-v_\ell}{N}})^{-1}
 \prod_{1\leq\ell  \leq k<m \leq 2k}(1- x^{-\delta_\ell} y^{-\delta_m} e^{\frac{-v_m-v_\ell}{N}})^{-1}
 \prod_{k< \ell \leq m \leq 2k}(1- y^{-\delta_m-\delta_\ell}e^{\frac{-v_m-v_\ell}{N}})^{-1}\nonumber \\
  & \times\prod_{1\leq \ell < m \leq k} (\delta_\ell v_\ell -\delta_m v_m)^2  \prod_{k< \ell < m \leq 2k} (\delta_\ell v_\ell -\delta_m v_m)^2\nonumber \\
  &\times \frac{e^{\sum_{j=1}^{2k} v_j} }{\prod_{j=1}^{k} v_j^k\left( 2\delta_j \log x\right)^k
  \left(\log^2x -\log^2 y \right)^k
  \prod_{j=k+1}^{2k} v_j^k\left( 2\delta_j \log y \right)^k
  \left(\log^2y -\log^2 x \right)^k}\nonumber \\
   &\times \mathrm{d}v_1\cdots \mathrm{d}v_{2k}\nonumber\\
  =&2^{-2k} (-1)^{k^2}   \left(\prod_{j=1}^{2k} \delta_j^{1-k}\right) x^{N\sum_{i=1}^k \delta_i} y^{N\sum_{i=k+1}^{2k} \delta_i}\nonumber \\&\times \prod_{1\leq \ell \leq m \leq k} (1- x^{-\delta_m-\delta_\ell}e^{\frac{-v_m-v_\ell}{N}})^{-1}
 \prod_{1\leq\ell  \leq k<m \leq 2k}(1- x^{-\delta_\ell} y^{-\delta_m} e^{\frac{-v_m-v_\ell}{N}})^{-1}
 \prod_{k< \ell \leq m \leq 2k}(1- y^{-\delta_m-\delta_\ell}e^{\frac{-v_m-v_\ell}{N}})^{-1}\nonumber\\
  & \times\prod_{1\leq \ell < m \leq k} (\delta_\ell v_\ell -\delta_m v_m)^2  \prod_{k< \ell < m \leq 2k} (\delta_\ell v_\ell -\delta_m v_m)^2 \frac{e^{\sum_{j=1}^{2k} v_j} }{\prod_{j=1}^{2k} v_j^k}\mathrm{d}v_1\cdots \mathrm{d}v_{2k}.
\end{align}
Let $a$ be the number of ``$+$'' signs among the $\delta_\ell$ corresponding to $x$ and $b$ the number of ``$+$'' signs among the $\delta_m$ corresponding to  $y$. The product over the reciprocal of linear terms involving $x$ above is
\begin{align}
 &\prod_{1\leq \ell \leq m \leq k} (1- x^{-\delta_m-\delta_\ell}e^{\frac{-v_m-v_\ell}{N}})^{-1}\nonumber \\
 &= \prod_{\substack{1\leq \ell \leq m \leq k\\ \delta_\ell=\delta_m=-1}} (1- x^2e^{\frac{-v_m-v_\ell}{N}})^{-1} (-1)^{\binom{a+1}{2}}\prod_{\substack{1\leq \ell \leq m \leq k\\ \delta_\ell=\delta_m=1}} x^2e^{\frac{v_m+v_\ell}{N}} (1- x^2e^{\frac{v_m+v_\ell}{N}})^{-1} \nonumber\\
 & \times \prod_{\substack{1\leq \ell \leq m \leq k\\ \delta_\ell\not =\delta_m}}  (1- e^{\frac{-v_m-v_\ell}{N}})^{-1}\nonumber\\
 =&(-1)^{\binom{a+1}{2}} x^{a(a+1)} \prod_{\substack{1\leq \ell \leq m \leq k\\ \delta_\ell\not =\delta_m}}  (1- e^{\frac{-v_m-v_\ell}{N}})^{-1} \prod_{\substack{1\leq \ell \leq m \leq k\\ \delta_\ell=\delta_m=1}} e^{\frac{v_m+v_\ell}{N}}  \prod_{\substack{1\leq \ell \leq m \leq k\\ \delta_\ell=\delta_m}} (1- x^2e^{\frac{\delta_m v_m+\delta_\ell v_\ell}{N}})^{-1}.\label{eq:x}
\end{align}
The terms involving $y$ are similar, whereas the mixed terms are
\begin{align}\label{eq:page19} 
 &\prod_{1\leq\ell  \leq k<m \leq 2k}(1- x^{-\delta_\ell} y^{-\delta_m} e^{\frac{-v_m-v_\ell}{N}})^{-1}\nonumber\\&=\prod_{\substack{1\leq\ell  \leq  k<m \leq 2k\\ \delta_\ell=\delta_m=-1 }}(1- x ye^{\frac{-v_m-v_\ell}{N}})^{-1}  (-1)^{ab}\prod_{\substack{1\leq\ell  \leq k<m \leq 2k\\ \delta_\ell=\delta_m=1 }}x ye^{\frac{v_m+v_\ell}{N}}(1- x ye^{\frac{v_m+v_\ell}{N}})^{-1}\nonumber\\
 & \times \prod_{\substack{1\leq\ell  \leq k<m \leq 2k\\ \delta_\ell=1, \delta_m=-1 }}(1- x^{-1} ye^{\frac{-v_m-v_\ell}{N}})^{-1}\prod_{\substack{1\leq\ell  \leq k<m \leq 2k\\ \delta_\ell=-1, \delta_m=1 }}(1- x y^{-1}e^{\frac{-v_m-v_\ell}{N}})^{-1}\nonumber\\
 &= (-1)^{ab}x^{ab}  y^{ab} \prod_{\substack{1\leq\ell  \leq k<m \leq 2k\\ \delta_\ell=\delta_m=1 }}
 e^{\frac{v_m+v_\ell}{N}}\prod_{\substack{1\leq\ell  \leq  k<m \leq 2k\\ \delta_\ell=\delta_m}}(1- x ye^{\frac{\delta_m v_m+\delta_\ell v_\ell}{N}})^{-1}\nonumber \\
 & \times \prod_{\substack{1\leq\ell  \leq k<m \leq 2k\\ \delta_\ell=1, \delta_m=-1 }}(1- x^{-1} ye^{\frac{-v_m-v_\ell}{N}})^{-1}\prod_{\substack{1\leq\ell  \leq k<m \leq 2k\\ \delta_\ell=-1, \delta_m=1 }}(1- x y^{-1}e^{\frac{-v_m-v_\ell}{N}})^{-1}.
\end{align}
The final product in \eqref{eq:page19} can be rewritten as a power series in $x^{-1}y$, so that it is equal to
\begin{equation*}
(-1)^{(k-a)b}x^{-(k-a)b}y^{(k-a)b}\prod_{\substack{1\leq\ell  \leq k<m \leq 2k\\ \delta_\ell=-1, \delta_m=1 }} e^{\frac{v_m+v_\ell}{N}} \prod_{\substack{1\leq\ell  \leq k<m \leq 2k\\ \delta_\ell=-1, \delta_m=1 }} (1-x^{-1}ye^{\frac{v_m+v_\ell}{N}})^{-1}.
\end{equation*}
Thus, \eqref{eq:page19}  becomes
\begin{align}
 & (-1)^{kb}x^{2ab-kb} y^{kb} \prod_{\substack{1\leq\ell  \leq k<m \leq 2k\\ \delta_m=1 }}
 e^{\frac{v_m+v_\ell}{N}}\prod_{\substack{1\leq\ell  \leq  k<m \leq 2k\\ \delta_\ell=\delta_m}}(1- x ye^{\frac{\delta_m v_m+\delta_\ell v_\ell}{N}})^{-1}\nonumber \\
 & \times \prod_{\substack{1\leq\ell  \leq k<m \leq 2k\\ \delta_\ell=1, \delta_m=-1 }}(1- x^{-1} ye^{\frac{-v_m-v_\ell}{N}})^{-1}\prod_{\substack{1\leq\ell  \leq k<m \leq 2k\\ \delta_\ell=-1, \delta_m=1 }} (1-x^{-1}ye^{\frac{v_m+v_\ell}{N}})^{-1}.\label{eq:xy}
\end{align}
Combining the lines \eqref{eq:x}, its $y$ version, and \eqref{eq:xy}, we can write the product of the reciprocal of linear terms involving $x$ and $y$ in \eqref{eq:page17} as
\begin{align}\label{eq:23}
 &(-1)^{\binom{a+1}{2}} x^{a(a+1)}\prod_{\substack{1\leq \ell \leq m \leq k\\ \delta_\ell\not =\delta_m}}  (1- e^{\frac{-v_m-v_\ell}{N}})^{-1} \prod_{\substack{1\leq \ell \leq m \leq k\\ \delta_\ell=\delta_m=1}} e^{\frac{v_m+v_\ell}{N}}  \prod_{\substack{1\leq \ell \leq m \leq k\\ \delta_\ell=\delta_m}} (1- x^2e^{\frac{\delta_m v_m+\delta_\ell v_\ell}{N}})^{-1}\nonumber\\
 &\times (-1)^{\binom{b+1}{2}} y^{b(b+1)}\prod_{\substack{k< \ell \leq m \leq 2k\\ \delta_\ell\not =\delta_m}}  (1- e^{\frac{-v_m-v_\ell}{N}})^{-1} \prod_{\substack{k< \ell \leq m \leq 2k\\ \delta_\ell=\delta_m=1}} e^{\frac{v_m+v_\ell}{N}}  \prod_{\substack{k<\ell \leq m \leq 2k\\ \delta_\ell=\delta_m}} (1- y^2e^{\frac{\delta_m v_m+\delta_\ell v_\ell}{N}})^{-1}\nonumber\\
 &\times (-1)^{kb}x^{2ab-kb} y^{kb} \prod_{\substack{1\leq\ell  \leq k<m \leq 2k\\ \delta_m=1 }}
 e^{\frac{v_m+v_\ell}{N}}\prod_{\substack{1\leq\ell  \leq  k<m \leq 2k\\ \delta_\ell=\delta_m}}(1- x ye^{\frac{\delta_m v_m+\delta_\ell v_\ell}{N}})^{-1}\nonumber \\
 & \times \prod_{\substack{1\leq\ell  \leq k<m \leq 2k\\ \delta_\ell \not = \delta_m}}(1- x^{-1} ye^{\frac{\delta_m(v_m+v_\ell)}{N}})^{-1}.
\end{align}
Our goal is to fix certain powers  $x^{(2c-k)N}$ and $y^{(2c-k)N}$ in \eqref{eq:page17} (which corresponds to $x^{(2c-k-2a+k)N}$ and $y^{(2c-k-2b+k)N}$ in \eqref{eq:23}) so that the final powers of $x$ and $y$ in \eqref{eq:Pk} are equal to $n = c2N$. Note that since all powers of $y$ in \eqref{eq:23} are positive, $b \le c$. There can be negative powers of $x$, but if we momentarily set $x = y$ in the above expression, our goal becomes to find the coefficient of $x^{(4a-2c-2b)N}$; in this case all powers of $x = y$ are nonnegative, so that we must also have $a + b \le 2c$. We will then collect the corresponding powers of $N$ in the coefficient as a function of $c$ by expanding all geometric series above, noting that if $0 < |y| < |x| < 1$, then all geometric series will converge. Taking asymptotics for $N$ large and expanding as in \cite{KR3}, \eqref{eq:23} becomes
\begin{align*}
\sim &(-1)^{\binom{a+b+1}{2}+(k-a)b} x^{a(a+b+1)-(k-a)b}\frac{N^{a(k-a)}}{\prod_{\substack{1\leq \ell \leq m \leq k\\ \delta_\ell\not =\delta_m}} (v_m+v_\ell)}  y^{b(a+b+1)+(k-a)b}\frac{N^{b(k-b)}}{\prod_{\substack{k< \ell \leq m \leq 2k\\ \delta_\ell\not =\delta_m}} (v_m+v_\ell)} \\
&\times \sum_{n=0}^\infty \sum_{\substack{\sum_{ 1\leq \ell \leq m \leq k} \beta_{\ell,m}=n\\ \beta_{\ell,m} \geq 0,  \delta_\ell=\delta_m\\ \beta_{\ell,m}= 0, \delta_\ell\not =\delta_m }} x^{2n} \exp \Big(\sum_{\substack{1\leq \ell \leq m \leq k\\\delta_\ell=\delta_m}} \beta_{\ell,m}\frac{\delta_m v_m+\delta_\ell v_\ell}{N}\Big)\\ 
 &\times  \sum_{n=0}^\infty \sum_{\substack{\sum_{ k<\ell \leq m \leq 2k} \beta_{\ell,m}=n\\ \beta_{\ell,m} \geq 0,  \delta_\ell=\delta_m\\ \beta_{\ell,m}= 0, \delta_\ell\not =\delta_m }} y^{2n} \exp \Big(\sum_{\substack{k< \ell \leq m \leq 2k\\\delta_\ell=\delta_m}} \beta_{\ell,m}\frac{\delta_m v_m+\delta_\ell v_\ell}{N}\Big)\\ 
 &\times \sum_{n=0}^\infty \sum_{\substack{\sum_{1\leq \ell\leq k < m \leq 2k} \beta_{\ell,m}=n\\ \beta_{\ell,m} \geq 0,  \delta_\ell=\delta_m\\ \beta_{\ell,m}= 0, \delta_\ell\not =\delta_m }} x^ny^{n} \exp \Big(\sum_{\substack{1\leq \ell\leq k < m \leq 2k\\\delta_\ell=\delta_m}} \beta_{\ell,m}\frac{\delta_m v_m+\delta_\ell v_\ell}{N}\Big)\\ 
 &\times \sum_{n=0}^\infty \sum_{\substack{\sum_{1\leq \ell\leq k < m \leq 2k} \beta_{\ell,m}=n\\ \beta_{\ell,m} \geq 0, \delta_\ell \ne \delta_m \\ \beta_{\ell,m} = 0, \delta_\ell = \delta_m}} x^{-n}y^{n} \exp \Big(\sum_{\substack{1\leq \ell\leq k < m \leq 2k\\\delta_\ell\ne \delta_m}} \beta_{\ell,m}\delta_m\frac{v_m
 +v_\ell}{N}\Big).\\
\end{align*}
The coefficient contributing to $x^{(2c-k) N}y^{(2c-k) N}$ coming from the four infinite sums above is 
\begin{align}\label{eq:Riemannsum}
&\sideset{}{'}\sum_{\substack{(\beta_{\ell,m})\\ 1 \le \ell \le m \le 2k}}\exp\Big(\frac 1N \sideset{}{'}\sum_{\ell,m} \beta_{\ell,m} \delta_m (v_\ell + v_m)\Big),
\end{align}
where the sum (and by an abuse of notation, the inner sum as well) is taken over all $\beta_{\ell,m}$ that appear in the four sums above, that is to say, all pairs $1\le \ell \le m \le 2k$ except for those where $1\le \ell \le m \le k$ or $k < \ell \le m \le 2k$ and $\delta_\ell \ne \delta_m$, and which satisfy the two constraints that
\begin{align}\label{eq:xconstraint}
&2\sum_{\substack{1 \le \ell \le m \le k \\ \delta_\ell = \delta_m}} \beta_{\ell,m} + \sum_{\substack{1 \le \ell \le k < m \le 2k \\ \delta_\ell = \delta_m}} \beta_{\ell,m} - \sum_{\substack{1 \le \ell \le k < m \le 2k \\ \delta_\ell \neq \delta_m}} \beta_{\ell,m}\nonumber \\&  = (2c-k -2a+k) N - a(a+b+1)+(k-a)b,\nonumber\\
\end{align}
and
\begin{align}\label{eq:yconstraint}
&2\sum_{\substack{k < \ell \le m \le 2k \\ \delta_\ell = \delta_m}} \beta_{\ell,m} + \sum_{\substack{1 \le \ell \le k < m \le 2k \\ \delta_\ell = \delta_m}} \beta_{\ell,m} + \sum_{\substack{1 \le \ell \le k < m \le 2k \\ \delta_\ell \neq \delta_m}} \beta_{\ell,m} \nonumber \\& = (2c-k-2b+k)  N - b(a+b+1)-(k-a)b.\nonumber\\
\end{align}

Change variables via 
\begin{equation*}
x_{\ell, m} = \frac{\beta_{\ell,m}}{N}.
\end{equation*}
As $N\rightarrow \infty$, the previous conditions become
\begin{align}\label{eq:xxconstraint-for-xs}
&2\sum_{\substack{1 \le \ell \le m \le k \\ \delta_\ell = \delta_m}} x_{\ell,m} + \sum_{\substack{1 \le \ell \le k < m \le 2k \\ \delta_\ell = \delta_m}} x_{\ell,m} - \sum_{\substack{1 \le \ell \le k < m \le 2k \\ \delta_\ell \neq \delta_m}} x_{\ell,m} = 2c-2a
\end{align}
and 
\begin{align}\label{eq:yyconstraint-for-xs}
&2\sum_{\substack{k < \ell \le m \le 2k \\ \delta_\ell = \delta_m}} x_{\ell,m} + \sum_{\substack{1 \le \ell \le k < m \le 2k \\ \delta_\ell = \delta_m}} x_{\ell,m} + \sum_{\substack{1 \le \ell \le k < m \le 2k \\ \delta_\ell \neq \delta_m}} x_{\ell,m} = 2c-2b,
\end{align}
where $x_{\ell,m} \geq 0$. Now interpret \eqref{eq:Riemannsum} as a Riemann sum, along the lines of \cite[Lemma 4.12]{KR3}. Since there are $2k^2 + k-a(k-a)-b(k-b)$ variables and two constraints, \eqref{eq:Riemannsum} is
\begin{align*}
&N^{2k^2 + k-a(k-a)-b(k-b) -2} \frac 1{N^{2k^2+k-a(k-a)-b(k-b)-2}}\sideset{}{'}\sum_{\substack{(\beta_{\ell,m})\\ 1 \le \ell \le m \le 2k}} \exp\Big(\sideset{}{'}\sum_{\ell,m} \frac{\beta_{\ell,m}}{N} \delta_m (v_\ell + v_m)\Big) \\
&= N^{2k^2 + k-a(k-a)-b(k-b) - 2} \sideset{}{'}\iint \exp\Big(\sideset{}{'}\sum_{\ell,m} x_{\ell,m} \delta_m (v_\ell + v_m)\Big) \sideset{}{'}\prod_{\ell,m} \mathrm{d}x_{\ell,m}\\& + O(N^{2k^2 + k-a(k-a)-b(k-b) - 3}).
\end{align*}
The integral over the $x_{\ell,m}$'s is really $(2k^2 + k - a(k-a) - b(k-b) - 2)$-dimensional; the variables can be taken to be, for example, all those $x_{\ell,m}$ appearing in \eqref{eq:xxconstraint-for-xs} or \eqref{eq:yyconstraint-for-xs} except for $x_{1,1}$ and $x_{2k,2k}$, which can be determined from the rest by \eqref{eq:xxconstraint-for-xs} or \eqref{eq:yyconstraint-for-xs}. Then the equations \eqref{eq:xxconstraint-for-xs} and \eqref{eq:yyconstraint-for-xs} continue to place size restraints on the remaining $x_{\ell,m}$. 

Finally, the coefficient of $x^{(2c-k) N}y^{(2c-k) N}$ in \eqref{eq:page17} is asymptotic to
\begin{align*}
\sim & 2^{-2k} (-1)^{(a+1)(k-b+1)-1+\binom{a+b+1}{2}}N^{2k^2 + k- 2} \oint \sideset{}{'}\iint \exp\Big(\sideset{}{'}\sum_{\ell,m} x_{\ell,m} \delta_m (v_\ell + v_m)\Big) \sideset{}{'}\prod_{\ell,m} \mathrm{d}x_{\ell,m}\\
&\times \frac{\prod_{1\leq \ell < m \leq k} (\delta_\ell v_\ell -\delta_m v_m)^2  \prod_{k< \ell < m \leq 2k} (\delta_\ell v_\ell -\delta_m v_m)^2}{\prod_{\substack{1\leq \ell \leq m \leq k\\ \delta_\ell\not =\delta_m}} (v_m+v_\ell)\prod_{\substack{k< \ell \leq m \leq 2k\\ \delta_\ell\not =\delta_m}} (v_m+v_\ell)} \times \frac{e^{\sum_{j=1}^{2k} v_j} }{\prod_{j=1}^{2k} v_j^k}\mathrm{d}v_1\cdots \mathrm{d}v_{2k}.
\end{align*}

Define the integral
\begin{equation*}
J({\bf v}) := \iint_{\substack{x_{\ell,m} \ge 0 \\ \eqref{eq:xxconstraint-for-xs}, \eqref{eq:yyconstraint-for-xs}}} \exp\left(\sideset{}{'}\sum x_{\ell,m} \delta_m (v_\ell + v_m)\right) \sideset{}{'}\prod_{\ell,m} \mathrm dx_{\ell,m}.
\end{equation*}
As ${\bf v} \to 0$, $J({\bf v})$ becomes a (complicated) polynomial in $c$ of degree $2k^2 + k - 2 - a(k-a) - b(k-b)$ (or identically 0), i.e. of degree the number of variables $x_{\ell,m}$ integrated in the definition of $J({\bf v})$. To see this, note that for any $n \ge 0$,
\begin{align*}
\lim_{v \to 0} \int_0^c x^n e^{xv} \mathrm dx 
&= \frac{c^{n+1}}{n+1},
\end{align*}
by integration by parts, L'H\^opital's rule, and differentiating under the integral.

Let us examine carefully what happens when $J({\bf v})$ is integrated in one variable; say, $x_{2,2}$. The coefficient of $x_{2,2}$ in the exponent is $2\delta_2 v_2 - 2 \delta_1 v_1$; the $v_1$ factor comes from \eqref{eq:xxconstraint-for-xs}. Thus after integrating in $x_{2,2}$, the integral is of the form 
\begin{align*}
J({\bf v}) &= \iint_{\substack{x_{\ell,m} \ge 0 \\ \eqref{eq:xxconstraint-for-xs}, \eqref{eq:yyconstraint-for-xs} \\ (\ell,m) \ne (2,2)}} (c - L({\bf x})) \exp\left(\sideset{}{'}\sum x_{\ell,m} \delta_m (v_\ell + v_m)\right) \sideset{}{'}\prod_{\ell,m} \mathrm dx_{\ell,m},
\end{align*}
where the sum and the product are taken over the remaining $(\ell,m)$ pairs and $L({\bf x})$ is a homogeneous linear function. When integrating in the next variable, we have either multiplied by a factor of $c$ or are integrating something of the form $xe^{xv}$ instead of the form $e^{xv}$; in either case, the degree has increased by one. Accordingly, at every step, the degree will increase by one as well; since we are integrating over $2k^2 + k - 2 - a(k-a)-b(k-b)$ variables, we end with a polynomial in $c$ of degree $2k^2 + k - 2 - a(k-a)- b(k-b)$ (or identically 0) as ${\bf v} \to 0$. Note that if instead of $J({\bf v})$ we multiplied the integrand by a polynomial in $x_{\ell,m}$ of degree $n$, we would instead end up with a polynomial in $c$ of degree  $2k^2 + k - 2 - a(k-a) - b(k-b) + n$ (or identically 0). 

The coefficient of $x^{(2c-k) N}y^{(2c-k) N}$ in \eqref{eq:page17} is thus asymptotically given by  
\begin{align}\label{eq:integral-with-J}
\sim & 2^{-2k} (-1)^{(a+1)(k-b+1)-1+\binom{a+b+1}{2}}N^{2k^2 + k- 2} \nonumber \\
&\times \oint J({\bf v}) \frac{\prod_{1\leq \ell < m \leq k} (\delta_\ell v_\ell -\delta_m v_m)^2  \prod_{k< \ell < m \leq 2k} (\delta_\ell v_\ell -\delta_m v_m)^2 }{\prod_{\substack{1\leq \ell \leq m \leq k\\ \delta_\ell\not =\delta_m}} (v_m+v_\ell)\prod_{\substack{k< \ell \leq m \leq 2k\\ \delta_\ell\not =\delta_m}} (v_m+v_\ell)} \frac{e^{\sum_{j=1}^{2k} v_j} \mathrm{d}v_1\cdots \mathrm{d}v_{2k}}{\prod_{j=1}^{2k} v_j^k}.
\end{align}
 Notice that 
\begin{equation}\label{eq:polynomial-in-complex-integral} \frac{\prod_{1\leq \ell < m \leq k} (\delta_\ell v_\ell -\delta_m v_m)^2 \prod_{k< \ell < m \leq 2k} (\delta_\ell v_\ell -\delta_m v_m)^2 }{\prod_{\substack{1\leq \ell \leq m \leq k\\ \delta_\ell\not =\delta_m}} (v_m+v_\ell)\prod_{\substack{k< \ell \leq m \leq 2k\\ \delta_\ell\not =\delta_m}} (v_m+v_\ell)} 
\end{equation}
is a polynomial of degree $2k(k-1)-a(k-a)-b(k-b)$. Therefore, to compute the integral \eqref{eq:integral-with-J}, we need to compute residues at $v_\ell=0$ for every $\ell$. This amounts to computing derivatives of $J({\bf v})$ (in the $v_\ell$'s). To find these residues, one has to differentiate $k-1$ times with respect to each $v_\ell$ and evaluate at $v_\ell = 0$. If the polynomial in  \eqref{eq:polynomial-in-complex-integral} is not differentiated exactly $2k(k-1)-a(k-a)-b(k-b)$ times, then it will give 0 upon evaluation at $v_\ell=0$.

Thus, the term $J({\bf v} )$ has to be differentiated $2k(k-1)-[2k(k-1)-a(k-a)-b(k-b)] = a(k-a)+b(k-b)$ times. The integrand of $J({\bf v})$ is of the form $\exp\left(\sum_{1 \le \ell \le 2k} v_\ell L_\ell({\bf x})\right)$, where $L_\ell({\bf x})$ is a certain linear function of the $x_{\ell,m}$, with no dependence on the $v_\ell$. Thus differentiating in any $v_\ell$ simply multiplies the integrand by some $L_\ell({\bf x})$. After differentiating $J({\bf v})$ a total of $a(k-a)+ b(k-b)$ times, we get an integral of the form
\begin{equation}\label{eq:derivative-of-J}
\iint_{\substack{x_{\ell,m} \ge 0 \\ \eqref{eq:xxconstraint-for-xs}, \eqref{eq:yyconstraint-for-xs}}} P({\bf x}) \exp\left(\sideset{}{'}\sum x_{\ell,m} \delta_m (v_\ell + v_m)\right) \sideset{}{'}\prod_{\ell,m} \mathrm dx_{\ell,m},
\end{equation}
where $P({\bf x})$ is a polynomial of degree $a(k-a)+b(k-b)$. But then as ${\bf v} \to 0$, the derivative of $J({\bf v})$ is a polynomial in $c$ of degree $2k^2 + k - 2 - a(k-a)-b(k-b) + a(k-a) + b(k-b) = 2k^2 + k - 2$.  

Thus by taking residues in \eqref{eq:integral-with-J}, the main term in \eqref{eq:Pk} can be written as
\[\gamma_{d_k,2}^S(c) (2N)^{2k^2 + k- 2},\]
where
\[\gamma_{d_k,2}^S(c)=\sum_{\substack{0\leq b\leq c \\ 0 \le a \le 2c-b}} g_{a,b}^S(c),\]
and each $g_{a,b}^S(t)$ is a polynomial of degree  $2k^2 + k - 2$.

Thus $\gamma_{d_k,2}^S(c)$ is a piecewise polynomial function of degree at most $2k^2+k-2$, and this concludes the proof of Theorem \ref{thm:complexsymplectic}.

\section{The Rudnick-Waxman Orthogonal setting}\label{sec:RW}
In this section we consider a problem involving an arithmetic sum in functions fields that is described by an integral over the ensemble of orthogonal matrices. More precisely,  we follow a model of Gaussian integers in the function field context that was considered by Rudnick and Waxman in \cite{Rudnick-Waxman} and initially developed by Bary-Soroker, Smilansky, and Wolf in \cite{BSSW}.

Let $\mathcal{M}$ be the set of monic polynomials in $\F_q[T]$, and let $\mathcal{M}_n$ denote the elements of $\mathcal{M}$  of degree $n$.  Similarly define $\mathcal{P}$ and $\mathcal{P}_n$ the set of monic irreducible polynomials, and the corresponding subset of elements of degree $n$. For a $P(T)\in \mathcal{P}$, there exist $A(T), B(T) \in \F_q[T]$ such that 
\begin{equation}\label{eq:norm}
P(T)=A(T)^2+TB(T)^2
\end{equation}
if and only if $P(0)$ is a square in $\F_q$. Let $\chi_2$ the character on $\F_q$ defined by 
\[\chi_2(x)=\begin{cases}
0 & x=0,\\
             1 & x \text{ is a nonzero square in } \F_q,\\
             -1 & \text{ otherwise.}
            \end{cases}\]
Extend $\chi_2$ to $\mathcal{M}$ by defining $\chi_2(f):=\chi_2(f(0))$. Equation \eqref{eq:norm} is solvable if and only if $\chi_2(P)=1$. 

Set $S:=\sqrt{-T}$ so that $\F_q[T]\subseteq \F_q[S]$ naturally. 
Equation \eqref{eq:norm} becomes 
\[P(T)=(A+BS)(A-BS)=\mathfrak{p}\overline{\mathfrak{p}}\]
 in  $\F_q[S]$. There are two automorphisms of $\F_q[S]$ fixing $\F_q[T]$: the identity and (an analogue of) complex conjugation, which may be extended to the ring of formal power series:
 \[\sigma:\F_q[[S]] \rightarrow \F_q[[S]], \qquad \sigma(S)=-S.\]
The norm map is then given by 
 \[\mathrm{Norm}: \F_q[[S]]^\times \rightarrow \F_q[[T]]^\times, \qquad \mathrm{Norm}(f)=f\sigma(f)=f(S)f(-S).\]
 The group of formal power series with constant term 1 and unit norm is  
\[\mathbb{S}^1:=\{g \in \F_q[[S]]^\times\, : \, g(0)=1, \mathrm{Norm}(g)=1\},\] 
and can be seen as an analogue of the unit circle in this context. 
 
 For $f \in  \F_q[[S]]$ let $\mathrm{ord}(f)=\max\{j \, :\, S^j \mid f\}$ and $|f|:=q^{-\mathrm{ord}(f)}$, the absolute value associated with the place at infinity. Now consider sectors of the unit circle:
\begin{equation}\label{eq:sector}
\mathrm{Sect}(v;k):=\{w\in \mathbb{S}^1\, :\, |w-v|\leq q^{-k}\}.
\end{equation}

The sector $\mathrm{Sect}(v;k)$ can be described modulo $S^k$ since
$w \in \mathrm{Sect}(v;k)$ if and only if  $w \equiv v \pmod{S^k}$. The following modular group then parametrizes the different sectors:
 \[\mathbb{S}^1_k:=\{f\in \F_q[S]/(S^k)\, : \, f(0)=1, \, \mathrm{Norm}(f)\equiv 1 \pmod{S^k}\}.\] 
\begin{lem}\cite[Lemma 2.1]{Katz},  \cite[Lemma 6.1]{Rudnick-Waxman},
\begin{enumerate}
\item  The cardinality of $\mathbb{S}^1_k$ is given by  \[\# \mathbb{S}^1_k=q^\kappa, \mbox{ with } \kappa: = \left \lfloor \frac{k}{2}\right \rfloor.\]  
\item There is a direct product decomposition
 \[\left(\F_q[S]/(S^k)\right)^\times = H_k\times \mathbb{S}^1_k,\]
 where 
 \[H_k:=\{ f\in \left(\F_q[S]/(S^k)\right)^\times  \, :\, f(-S) \equiv f(S) \pmod{S^k}\},\]
 and 
 \[\# H_k=(q-1)q^{\left\lfloor \frac{k-1}{2}\right \rfloor}.\]
\end{enumerate}
 \end{lem}

For $f \in \F_q[S]$ coprime to $S$, define
 \[U(f):=\sqrt{ \frac{f}{\sigma(f)}}.\]
The square-root is well defined for 
 $\frac{f}{\sigma(f)}\in \mathbb{S}^1$, since  $v \mapsto v^2$ is an automorphism of $\mathbb{S}^1$ by Hensel's Lemma. Also, $U(cf)=U(f)$ for scalars $c \in \F_q^\times$. 
 
The modular counterpart of $U$ is given by
 \[U_k: \left(\F_q[S]/(S^k)\right)^\times \rightarrow \mathbb{S}^1_k, \qquad f \mapsto \sqrt{\frac{f}{\sigma(f)}} \pmod{S^k}\]
and is a surjective homomorphism whose kernel is $H_k$ (\cite[Lemma 6.2]{Rudnick-Waxman}).

We now define the analogues of the Hecke characters in this setting, following \cite{Rudnick-Waxman} and Katz \cite{Katz}. A \emph{super-even} character modulo $S^k$ is a Dirichlet character 
 \[\Xi: \left(\F_q[S]/(S^k)\right)^\times \rightarrow \C^\times\]
 which is trivial on $H_k$. Super-even characters modulo $S^k$ can be seen as the characters of 
 $\left(\F_q[S]/(S^k)\right)^\times /H_k\cong \mathbb{S}_k^1$.
 
 The super-even characters modulo $S$ satisfy these orthogonality relations, found in the proof of Lemma 6.8 in \cite[page~192]{Rudnick-Waxman}:
\begin{equation}\label{eq:orthogonalitysupereven}
\sum_{v \in \mathbb{S}_k^1} \overline{\Xi_1(v)}\Xi_2(v)=\begin{cases}
                                                           q^\kappa & \Xi_1=\Xi_2,\\
                                                           0 & \text{otherwise}.
                                                          \end{cases}
                                \end{equation}

\begin{prop}{\cite[Proposition 6.3]{Rudnick-Waxman}} \label{prop:6.3}
For $f\in \left(\F_q[S]/(S^k)\right)^\times$ and $v \in \mathbb{S}_k^1$, the following are equivalent:
 \begin{enumerate}
  \item $U_k(f) \in \mathrm{Sect}(v;k)$,
  \item $U_k(f)=U_k(v)$,
  \item $fH_k=vH_k$,
  \item $\Xi(f)=\Xi(v)$ for all super-even characters $\pmod{S^k}$. 
 \end{enumerate}
 \end{prop}

 The \emph{Swan conductor} of $\Xi$ is the maximal integer $d=d(\Xi)<k$ such that $\Xi$ is nontrivial on the subgroup
 \[\Gamma_d:=\left(1+(S^d)\right)/(S^k)\subset \left(\F_q[S]/(S^k)\right)^\times .\] 
Thus $\Xi$ is a primitive character modulo $S^{d(\Xi)+1}$. The Swan conductor of a super-even character is necessarily odd, since these characters are automatically trivial on $\Gamma_d$ for $d$ even. 

Consider the twists $\Xi\chi_2$ of the super-even characters $\Xi$. The 
$L$-function associated to $\Xi \chi_2$ is given by 
\[\mathcal{L}(u, \Xi\chi_2)=\sum_{\substack{f \in \mathcal{M}\\f(0)\not = 0}} \Xi(f)\chi_2(f) u^{\deg(f)}=\prod_{\substack{P\in \mathcal{P}\\P(0)\not = 0}} \left( 1-\Xi(P)\chi_2(P) u^{\deg(P)}\right)^{-1},\quad |u|<1/q.\]
This is a polynomial of degree $d(\Xi)$ when $\Xi$ is nontrivial. The Riemann hypothesis, which is known in this case, implies that the reciprocals of the roots of $\mathcal{L}(u,\Xi\chi_2)$ have absolute value $q^{1/2}$. More precisely,
 \[\mathcal{L}(u, \Xi\chi_2)=\det\left(I-uq^{1/2}\Theta_{\Xi\chi_2} \right),\]
 with $\Theta_{\Xi\chi_2} \in \mathrm{U}(N)$ and $N=d(\Xi)$. 
 
 Katz \cite[Theorem 7.1]{Katz}  showed that for $q\rightarrow \infty$ the set of Frobenius classes 
 \[\{\Theta_{\Xi\chi_2}\, :\, \Xi \mbox{ primitive super-even } \pmod{S^k}\}\]
 becomes uniformly distributed in $\mathrm{O}(2\kappa-1)$ if $2\kappa  \geq 6$, and that the same holds for $2\kappa = 4$ if the characteristic is coprime to 10. 
 
 \subsection{A problem leading to orthogonal distributions}
 Our goal is to study the following sum:
\begin{equation*}
\mathcal N_{d_\ell,k,n}^O(v) = \sum_{\substack{f \in \mathcal M_{n} \\ f(0) \ne 0 \\ U(f) \in \mathrm{Sect}(v,k)}} d_\ell(f) \left(\frac{1 + \chi_2(f)}{2}\right),\end{equation*}
and to prove Theorem \ref{thm:ortho-intro} about its variance.
 
The divisor function naturally appears when considering powers of the $L$-function, since
\[\mathcal L(u, \Xi\chi_2)^\ell  = \mathrm{det}\left(1 - uq^{1/2}\Theta_{\Xi\chi_2}\right)^\ell = \sum_{\substack{f \in \mathcal M\\f(0)\not = 0}} d_\ell (f)\Xi(f)\chi_2(f)u^{\deg(f)}.\]
This suggests that we define
\[M_0(n;d_\ell \Xi\chi_2) := \sum_{\substack{f \in \mathcal M_{n} \\f(0)\not =0}} d_\ell(f) \Xi(f) \chi_2(f).\]
A Dirichlet character $\chi$ on $\F_q[T]$ is \emph{even} if $\chi(\alpha f)=\chi(f)$ for all $\alpha \in \F_q^\times$ and odd otherwise. In our case, $\Xi$ is even while $\Xi\chi_2$ is odd. 

By \cite[Lemma 2.1]{KuperbergLalin}, the following result holds. 
\begin{lem}\label{lem:Mtwisted}
 For $n \leq \ell d(\Xi)$,
 \[M_0(n;d_\ell \Xi\chi_2)=(-1)^nq^\frac{n}{2} \sum_{\substack{j_1+\cdots +j_\ell=n\\0\leq j_1,\dots j_\ell \leq d(\Xi) }}\mathrm{Sc}_{j_1}(\Theta_{\Xi\chi_2})\cdots \mathrm{Sc}_{j_\ell}(\Theta_{\Xi\chi_2})\]
and $M_0(n;d_\ell \Xi\chi_2)=0$ otherwise.
\end{lem}




We start our analysis by looking at the mean value of $\mathcal N_{d_\ell,k,n}^O$ averaged over all the directions of $v\in \mathbb{S}_k^1$. 
\begin{lem} \label{lem:meanNo}Define 
\[\langle \mathcal{N}^O_{d_\ell,k,n} \rangle :=\frac{1}{q^\kappa} \sum_{v \in \mathbb{S}_k^1}  \mathcal{N}^O_{d_\ell,k,n}(v).\]
Then
\[\langle \mathcal{N}^O_{d_\ell,k,n}\rangle=\frac{1}{q^\kappa}\sum_{\substack{f \in \mathcal{M}_n\\ f(0)\not =0}} d_\ell(f)\left(\frac{1 + \chi_2(f)}{2}\right)=\frac{q^{n-\kappa}}{2} \binom{\ell+n-1}{\ell-1} +O(q^{n-\kappa-1}).\]
\end{lem}
\begin{proof}
 By \cite[Lemma 5.5]{KuperbergLalin},
 \[\frac{1}{q^\kappa}\sum_{\substack{f \in \mathcal{M}_n\\ f(0)\not =0}} d_\ell(f)=q^{n-\kappa}\binom{\ell+n-1}{\ell-1} +O(q^{n-\kappa-1}).\]
 This can be seen by considering the generating function for $\sum_{\substack{f \in \mathcal{M}\\ f(0)\not =0}}d_\ell(f) u^{\deg(f)}$, namely $\left(\frac{1-u}{1-qu}\right)^\ell$, and by comparing the coefficients of $u^n$ in both expressions.
 
 The result follows from the fact that
 \begin{equation}\label{eq:smallsumd}\sum_{\substack{f \in \mathcal M_{n} \\ f(0) \ne 0}}d_\ell(f)\chi_2(f) = \begin{cases} 0 &\text{ if } n > 0, \\ 1 &\text{ if } n = 0.\end{cases}\end{equation}
This fact is true because the generating function of the terms in \eqref{eq:smallsumd} is given by 
 \[\sum_{\substack{f \in \mathcal M \\ f(0) \ne 0}}d_\ell(f)\chi_2(f) u^{\deg(f)}=\Big(\sum_{\substack{f \in \mathcal M \\ f(0) \ne 0}}\chi_2(f)  u^{\deg(f)}\Big)^\ell,\]
and at the same time, 
\[\sum_{\substack{f \in \mathcal M_n \\ f(0) \ne 0}}\chi_2(f) =\begin{cases} 0 &\text{ if } n > 0, \\ 1 &\text{ if } n = 0. \end{cases}\]
\end{proof}

Note that in proving Lemma \ref{lem:meanNo}, we also showed that $\langle \mathcal N_{d_\ell,k,n}^O\rangle$ is the average sum of the divisor function over all sectors, unweighted by $\chi_2$, and then divided by the number of sectors considered. That is,
\begin{equation*}
\langle \mathcal N_{d_\ell, k, n}^O\rangle = \frac 1{2q^\kappa} \sum_{\substack{f \in \mathcal M_n \\ f(0) \ne 0}} d_\ell(f).
\end{equation*}
However, the quantity $\mathcal N_{d_\ell,k,n}^O(v)$ exhibits two kinds of fluctuations away from this average: one is simply the fluctuations of the divisor function itself in sectors, and the other is the fluctuations coming from the twist by the character $\chi_2$. We are interested in isolating the latter fluctuations, which are somewhat more delicate, so we will consider a more refined ``average'' of $\mathcal N_{d_\ell,k,n}^O$, defined as:
\begin{equation*}
\langle \mathcal N_{d_\ell,k,n}^O(v)\rangle_S := \frac 12 \sum_{\substack{f \in \mathcal M_n \\ f(0) \ne 0 \\ U(f) \in \mathrm{Sect}(v,k)}} d_\ell(f).
\end{equation*}

Our next goal is to obtain $\mathcal N_{d_\ell,k,n}^O$ in terms of the characters $\Xi$ and $\chi_2$. By Proposition \ref{prop:6.3} and the orthogonality relations \eqref{eq:orthogonalitysupereven}, for $f \in \mathcal{M}_n$, 
\begin{equation*}
\frac{1}{q^{\kappa}} \sum_{\Xi \text{ super-even} \pmod{S^k}} \overline{\Xi(v)}\Xi(f) =\begin{cases} 1 & U(f)\in \mathrm{Sect}(v;k),\\ 0 & \text{otherwise}.
\end{cases}
\end{equation*}
Hence, we can write
\begin{align*}
\mathcal{N}^O_{d_\ell,k,n}(v)=&\sum_{\substack{f \in \mathcal{M}_n\\ f(0)\not =0\\U(f)\in \mathrm{Sect}(v,k)}} d_\ell(f)\left(\frac{1 + \chi_2(f)}{2}\right) \\
=& \frac 12 \sum_{\substack{f \in \mathcal{M}_n\\ f(0)\not =0\\U(f)\in \mathrm{Sect}(v,k)}} d_\ell(f) + \frac 12 \sum_{\substack{f \in \mathcal{M}_n\\ f(0)\not =0\\U(f)\in \mathrm{Sect}(v,k)}} d_\ell(f)\chi_2(f) \\
=&\frac{1}{2q^{\kappa}}\sum_{\Xi \text{ super-even} \pmod{S^k}} \overline{\Xi(v)} \sum_{\substack{f \in \mathcal{M}_n\\ f(0)\not =0}} d_\ell(f)\Xi(f)\\& + \frac{1}{2q^{\kappa}}\sum_{\Xi \text{ super-even} \pmod{S^k}} \overline{\Xi(v)} \sum_{\substack{f \in \mathcal{M}_n\\ f(0)\not =0}} d_\ell(f)\Xi(f)\chi_2(f).
\end{align*}
Isolating the contribution from the trivial character $\Xi_0$, yields
\begin{align*}
\mathcal{N}^O_{d_\ell,k,n}(v)
&= \frac {1}{2q^\kappa} \sum_{\substack{f \in \mathcal{M}_n\\ f(0)\not =0\\}} d_\ell(f) + \frac  {1}{2q^\kappa} \sum_{\substack{f \in \mathcal M_{n} \\ f(0) \ne 0}}d_\ell(f)\chi_2(f) \\
&+  \frac{1}{2q^{\kappa}}\sum_{\substack{\Xi \text{ super-even} \pmod{S^k} \\ \Xi \ne \Xi_0}} \overline{\Xi(v)} \sum_{\substack{f \in \mathcal{M}_n\\ f(0)\not =0}} d_\ell(f)\Xi(f)\\
&+  \frac{1}{2q^{\kappa}}\sum_{\substack{\Xi \text{ super-even} \pmod{S^k} \\ \Xi \ne \Xi_0}} \overline{\Xi(v)} \sum_{\substack{f \in \mathcal{M}_n\\ f(0)\not =0}} d_\ell(f)\Xi(f)\chi_2(f)\\
&= \frac {1}{2q^\kappa} \sum_{\substack{f \in \mathcal{M}_n\\ f(0)\not =0\\}} d_\ell(f) + \frac  {1}{2q^\kappa} \sum_{\substack{f \in \mathcal M_{n} \\ f(0) \ne 0}}d_\ell(f)\chi_2(f) \\
& + \frac{1}{2q^{\kappa}}\sum_{\substack{\Xi \text{ super-even} \pmod{S^k} \\ \Xi \ne \Xi_0}} \overline{\Xi(v)} M_0(n;d_\ell \Xi)+ \frac{1}{2q^{\kappa}}\sum_{\substack{\Xi \text{ super-even} \pmod{S^k} \\ \Xi \ne \Xi_0}} \overline{\Xi(v)} M_0(n;d_\ell \Xi \chi_2).
\end{align*}
By Lemma \ref{lem:meanNo}, the first term above corresponds to $\langle \mathcal{N}^O_{d_\ell,k,n}\rangle$, and the second term above is $0$. The sum of the first and third terms, meanwhile, is the more refined average $\langle \mathcal N^O_{d_\ell,k,n}\rangle_S$. In particular,
\begin{align}
 \mathcal{N}^O_{d_\ell,k,n}(v)- \langle \mathcal{N}^O_{d_\ell,k,n}(v)\rangle_S = 
& \frac{1}{2q^{\kappa}}\sum_{\substack{\Xi \text{ super-even} \pmod{S^k} \\ \Xi \ne \Xi_0}} \overline{\Xi(v)} M_0(n;d_\ell \Xi \chi_2). \label{eq:sumM0} 
\end{align}

Recall that our goal is to compute the variance
\begin{align}\label{eq:var-S-dk}
 \mathrm{Var}_S(\mathcal{N}^O_{d_\ell,k,n})=& \frac{1}{q^\kappa}\sum_{v \in \mathbb{S}_k^1} \left|\mathcal{N}^O_{d_\ell,k,n}(v)- \langle \mathcal{N}^O_{d_\ell,k,n}(v)\rangle_S\right|^2.
\end{align}

By applying the orthogonality relations \eqref{eq:orthogonalitysupereven}
 to equations \eqref{eq:sumM0} and \eqref{eq:var-S-dk}, we obtain
 \begin{align*}
 \mathrm{Var}_S(\mathcal{N}^O_{d_\ell,k,n})= 
&\frac{1}{q^\kappa}\sum_{v \in \mathbb{S}_k^1}
\frac{1}{4q^{2\kappa}}\sum_{\substack{\Xi_1, \Xi_2 \text{ super-even} \pmod{S^k}\\ \Xi_1, \Xi_2\not = \Xi_0}}  \overline{\Xi_1(v)} M_0(n; d_\ell \Xi_1\chi_2) \Xi_2(v) \overline{M_0(n; d_\ell \Xi_2\chi_2)}\\
=&
\frac{1}{4q^{2\kappa}}\sum_{\substack{\Xi_1, \Xi_2 \text{ super-even} \pmod{S^k}\\ \Xi_1, \Xi_2\not = \Xi_0}}  M_0(n; d_\ell \Xi_1\chi_2)  \overline{M_0(n; d_\ell \Xi_2\chi_2)}
 \frac{1}{q^\kappa}\sum_{v \in \mathbb{S}_k^1} \overline{\Xi_1(v)} \Xi_2(v) \\
=&\frac{1}{4q^{2\kappa}}\sum_{\substack{\Xi  \text{ super-even} \pmod{S^k}\\ \Xi \not = \Xi_0}} |M_0(n; d_\ell \Xi\chi_2) |^2.
\end{align*}                                   

Combining this with Lemma \ref{lem:Mtwisted} yields 
\begin{equation}\label{eq:varsuperevenortho}\mathrm{Var}_S(\mathcal N^O_{d_\ell,k,n}) =\frac{q^n}{4q^{2\kappa}} \sum_{\substack{\Xi \text{ super-even} \pmod{S^k} \\ \Xi \ne \Xi_0}} \Big| \sum_{\substack{j_1 + \cdots + j_\ell = n \\ 0 \le j_1, \dots, j_\ell \le d(\Xi) }} \mathrm{Sc}_{j_1}(\Theta_{\Xi\chi_2}) \cdots \mathrm{Sc}_{j_\ell}(\Theta_{\Xi\chi_2})\Big|^2.
\end{equation}

We are now ready to prove the main result of this section. 
\begin{thm}
Let $n\leq \ell (2\kappa-1)$ with $\kappa =\lfloor \frac{k}{2}\rfloor\geq 3$. As $q \rightarrow \infty$,   
\[\mathrm{Var}_S(\mathcal N^O_{d_\ell,k,n}) \sim  \frac{q^n}{4q^{\kappa}} \int_{\mathrm{O}(2\kappa - 1)} \Big| \sum_{\substack{j_1 + \cdots + j_\ell = n \\ 0 \le j_1, \dots, j_\ell \le 2\kappa - 1}} \mathrm{Sc}_{j_1}(U) \cdots \mathrm{Sc}_{j_\ell}(U)\Big|^2 \mathrm{d}U. \]
\end{thm}
\begin{proof}
In equation \eqref{eq:varsuperevenortho}, separate the characters according to their Swan conductor, which is necessarily an odd integer $d(\Xi)<k$ with maximal value $2\kappa-1$. The characters with maximal conductor are primitive; the contribution from the others is negligible. Thus, we can consider the sum only over the primitive characters, and the  result follows from Katz \cite[Theorem 7.1]{Katz}.
\end{proof}

\begin{rem}
In fact, the sum analogous to $\mathcal N^O_{d_\ell,k,n}$ defined with $\left(\frac{1-\chi_2(f)}{2}\right)$ instead of $\left(\frac{1+\chi_2(f)}{2}\right)$ has the same variance by the same arguments as in this section. Similarly, the distribution of the divisor function along quadratic non-residues modulo an irreducible polynomial is symplectic, just as the distribution along quadratic residues is.
\end{rem}

\section{The symmetric function theory context in the  orthogonal case} \label{sec:symmetric-orthogonal}
In this section we describe a model for $I_{d_k,2}^O(n;N)$ using symmetric function theory. Several ideas translate from the symplectic case treated in Section \ref{sec:symmetric-symplectic}.

We begin with a lemma connecting orthogonal matrix integrals to Schur functions. The \emph{conjugate partition} $\mu'$ of a partition $\mu$ is defined by the condition that the Ferrer diagram  of $\mu'$ is the transpose of the Ferrer diagram of $\mu$. We have the following expression. 
\begin{lem}
 \begin{equation}\label{eq:trulyfinalsoso}\int_{\mathrm{O}(2N+1)} \prod_{j=1}^r \det(1+Ux_j) \mathrm{d}U
=\sum_{\substack{\mu_1 \le 2N+1 \\ \mu' \text{ even}}} s_{\mu}(x_1, \dots, x_r),
\end{equation}
where the sum takes place over all the partitions with even conjugate partition.  
\end{lem}
\begin{rem} The above result is very similar to equation (102) in \cite{Bump-Gamburd}, which gives 
\begin{equation*}
\int_{\mathrm{O}(2N)} \prod_{j=1}^r \det(1+Ux_j) \mathrm{d}U =\sum_{\substack{\lambda_1\leq 2N\\\lambda' \,\text{even}}}s_\lambda(x_1,\dots,x_r).
\end{equation*}
Their results do not immediately extend to the orthogonal group of odd dimension. However, this result is nevertheless well-known to experts; for completeness, we provide a short proof.
\end{rem}
\begin{proof}
We begin by using the dual Cauchy identity (see \cite[Theorem 7.14.3, page 332]{Stanley-enumerative}), which says that for any vectors $(\alpha_1, \dots, \alpha_n)$ and $(\beta_1, \dots, \beta_m)$, 
\begin{equation*}
\prod_{i=1}^n \prod_{j=1}^m (1 + \alpha_i\beta_j) = \sum_{\mu} s_{\mu}(\alpha_1, \dots, \alpha_n) s_{\mu'}(\beta_1, \dots, \beta_m),
\end{equation*}
where the sum takes place over all partitions $\mu$. Then if $U$ has eigenvalues $\alpha_1, \dots, \alpha_n$, we have
\begin{align*}
\prod_{j=1}^r \det(1 + Ux_j) &= \prod_{j=1}^r \prod_{i=1}^n (1+ \alpha_ix_j) \\
&= \sum_{\mu} s_{\mu}(x_1, \dots, x_r) s_{\mu'}(\alpha_1, \dots, \alpha_n).
\end{align*}
Integrating this identity gives that
\begin{align*}
\int_{\mathrm{O}(2N+1)} \prod_{j=1}^r \det(1+Ux_j) \mathrm{d}U &= \sum_{\mu} s_{\mu}(x_1, \dots, x_r) \int_{\mathrm O(2N+1)} s_{\mu'}(\alpha_1, \dots, \alpha_n) \mathrm dU.
\end{align*}
By \cite[equation (56)]{Diaconis-Gamburd} (see also the arguments leading to \cite[equation 3.21]{Macdonald} for more explanation), 
\begin{equation*}
\int_{\mathrm O(2N+1)} s_{\mu'}(\alpha_1, \dots, \alpha_n) \mathrm dU = \begin{cases} 1 &\text{ if } \mu' \text{ is even and } \ell(\mu') \le 2N + 1 \\
0 &\text{ otherwise.} \end{cases}
\end{equation*}
Noting that $\ell(\mu') \le 2N+1$ if and only if $\mu_1 \le 2N+1$, this completes the proof.
\end{proof}

In our case,
\begin{equation}\label{eq:geno}
\int_{\mathrm{O}(2N+1)} \det(1+Ux)^k\det(1+Uy)^k \mathrm{d}U= \sum_{m,n=0}^{(2N+1)k}x^m y^nJ_{d_k,2}^O(m,n;N),
\end{equation}and we are interested in the diagonal terms 
\begin{equation*}
I_{d_k,2}^O(n;N):=J_{d_k,2}^O(n,n;N).
\end{equation*}
Thus
\begin{equation*}
I_{d_k,2}^O(n;N) = 2\left[\sum_{\substack{\mu_1 \le 2N+1 \\ \mu'\text{ even}}} s_{\mu}(\underbrace{x,\dots, x}_k, \underbrace{y,\dots, y}_k)\right]_{x^ny^n}.
\end{equation*}
These terms count the number of SSYTs $T$ such that $\mu_1 \le 2N+1$, $\mu'$ is even, and 
\[a_1+\cdots +a_k=n, \qquad a_{k+1}+\cdots+a_{2k}=n.\]
Parametrize such $T$ using the same parametrization as before, where $y^{(s)}_r = y^{(s)}_r(T)$ is the rightmost position of the entry $s$ in the row $r$. We obtain an array of shape   
\[\begin{bmatrix}
   y_1^{(1)} &  y_1^{(2)} &  \dots &  y_1^{(2k)}\\
    0&  y_2^{(2)} &  \dots &  y_2^{(2k)}\\
   \vdots & \ddots & \ddots & \vdots  \\
   0 & 0 & \cdots &  y_{2k}^{(2k)}
  \end{bmatrix}\]
and satisfying the following properties: 
\begin{enumerate}
 \item $y_r^{(s)} \in \Z\cap [0,2N+1]$, 
 \item $y_{r+1}^{(s+1)}\leq y_r^{(s)}$,
 \item $y_r^{(s)}\leq y_r^{(s+1)}$,
 \item $y_1^{(k)}+\cdots+y_k^{(k)}=n$,
 \item $y_1^{(2k)}+\cdots+y_{2k}^{(2k)}=2n$, and
 \item $y_{2\ell-1}^{(2k)} = y_{2\ell}^{(2k)}$ for $\ell=1,\dots,k$. 
\end{enumerate}
The last condition codifies the fact that $\mu'$ is even.
The combination of (6) with items (2) and (3) imposes the extra conditions $y_{2\ell-1}^{(2k-1)} =y_{2\ell-1}^{(2k)}$ for $\ell=1,\dots, k$. 

Let $c =\frac{n}{2N+1}$. As in the symplectic case, consider $U^O=\{(i,j): 1\leq i \leq j \leq 2k-1, (i,j)\not = (1,k), (1,2k-1)\}$. Let $V^O_c$ be the convex region contained in $\R^{2k^2-k-2}=\{(u_i^{(j)})_{(i,j)\in U^O} : u_i^{(j)} \in \R\}$ defined by the following inequalities:
\begin{enumerate}
 \item $0\leq u_i^{(j)} \leq 1$ for $(i,j) \in U^O$,
 \item for $u_1^{(k)}:=c-(u_2^{(k)}+\cdots+u_k^{(k)})$, we have $0\leq u_1^{(k)}\leq 1$,
   \item for $u_1^{(2k-1)}:=c-(u_3^{(2k-1)}+\cdots+u_{2k-1}^{(2k-1)})$, we have $0\leq u_1^{(2k-1)}\leq 1$,
   \item $u_{r+1}^{(s+1)}\leq u_r^{(s)}$, and
 \item $u_r^{(s)}\leq u_r^{(s+1)}$.
\end{enumerate}
Let $A_k^O$ be the set defined by the last two conditions above. 

As in the symplectic case, the region $V^O_c$ is convex because it is the intersection of half-planes. For $c \in [0,k]$, $V^O_c$ is contained in $[0,1]^{2k^2-k-2}$ and is therefore contained in a closed ball of diameter $\sqrt{2k^2-k-2}$. 
Then
\begin{equation*}
I_{d_k,2}^O(n;N)=2\# \left(\Z^{2k^2-k-2} \cap ((2N+1)\cdot V^O_c) \right),
\end{equation*}
where $(2N+1)\cdot V^O_c=\{ (2N+1)x: x\in V^O_c\}$ is the dilate of $V^O_c$ by a factor of $2N+1$. 
By Theorem \ref{lem:schmidt},
\[I_{d_k,2}^O(n;N)= 2\mathrm{vol}_{2k^2-k-2} ((2N+1)\cdot V^O_c)+O\left(N^{2k^2-k-3}\right).\]
The volume of the corresponding region is given by 
\begin{equation}\label{eq:volume-orth}
\begin{aligned}
 & \mathrm{vol}_{2k^2-k-2} ((2N+1)\cdot V^O_c)\\&=(2N+1)^{2k^2-k-2} \int_{[0,1]^{2k^2-k}} \delta_c(u_1^{(k)}+\cdots+u_k^{(k)})\delta_{c}(u_1^{(2k-1)}+u_3^{(2k-1)}+\cdots +u_{2k-1}^{(2k-1)}) \mathds{1}_{A_k^O}(u)
  d^{2k^2-k} u.
  \end{aligned}
\end{equation}

By applying Theorem \ref{thm:Ehrhart}, we obtain Theorem \ref{thm:degreeEhrharto}, which is an analogue to Theorem  \ref{thm:degreeEhrharts}.
\begin{cor} 
 Let $c=\frac{a}{b}$ be a fixed rational number and $k$ be a fixed integer. If $2N+1$ is a multiple of $b$, then $I_{d_k,2}^O(c(2N+1);N)= P_{c,k}(N)$, where $P_{c,k}$ is a polynomial of degree ${2k^2-k-2}$.
\end{cor}
We also have the analogue to Proposition \ref{eq:intsymp}, with a formula for $\gamma_{d_k,2}^O$ itself.
\begin{prop}The coefficient for $(2N+1)^{2k^2-k-2}$ in the asymptotics for $I_{d_k,2}^O(c(2N+1),N)$ is given by 
 \begin{align*}
 \gamma_{d_k,2}^O(c)=&\frac{2}{G(1+k)}\int_{[0,1]^{\frac{3}{2}k^2-\frac{k}{2}}} \delta_c(u_1^{(k)}+\cdots+u_k^{(k)}) \delta_{c}(u_1^{(2k-1)}+\cdots+u_{2k-1}^{(2k-1)})\\
&\times \mathds{1}   \begin{bmatrix}
    &   & u_1^{(k)} & \leq &   \dots &\leq & u_1^{(2k)}\\
    &   \iddots & &   & & \iddots& \\
     u_k^{(k)} & \leq & \dots & \leq & u_k^{(2k)} &\\
   \vdots & &\iddots  & &  &&   \\
   u_{2k}^{(2k)} && & &&&
  \end{bmatrix}
\Delta(u_1^{(k)},u_2^{(k)},\dots,u_k^{(k)}) d^{\frac{3}{2}k^2-\frac{k}{2}} u.
\end{align*}
\end{prop}
\begin{proof}
From the volume formula \eqref{eq:volume-orth}, we have
\begin{align}\label{eq:intgamma-orth}
 \gamma_{d_k,2}^O(c)=&2\int_{[0,1]^{2k^2-k}} \delta_c(u_1^{(k)}+\cdots+u_k^{(k)}) \delta_{c}(u_1^{(2k-1)}+u_3^{(2k-1)}+\cdots+u_{2k-1}^{(2k-1)})\nonumber\\
&\times \mathds{1} \begin{bmatrix}
   u_1^{(1)} & \leq & u_1^{(2)} & \leq & \dots & \leq &   u_1^{(2k-2)} &\leq & u_1^{(2k-1)}\\
    \vertgeq  & &   \vertgeq   & & & &   \vertgeq  & & \\
     u_2^{(2)} & \leq & u_2^{(3)} & \leq & \dots & \leq & u_2^{(2k-1)} &&\\
   \vdots & \vdots & \vdots & \vdots &\iddots  & && &   \\
   u_{2k-1}^{(2k-1)} && & &&&&&
  \end{bmatrix}
 d^{2k^2-k} u.
\end{align}
Integrating with respect to $u_1^{(1)}$ in \eqref{eq:intgamma-orth} and applying Lemma \ref{lem:KR3}, 
\begin{align*}
 \gamma_{d_k,2}^O(c)=&2\int_{[0,1]^{2k^2-k-1}} \delta_c(u_1^{(k)}+\cdots+u_k^{(k)}) \delta_{c}(u_1^{(2k-1)}+u_3^{(2k-1)}+\cdots+u_{2k-1}^{(2k-1)})\\
&\times \mathds{1}  \begin{bmatrix}
    & & u_1^{(2)} & \leq & \dots & \leq &   u_1^{(2k-2)} &\leq & u_1^{(2k-1)}\\
    & &   \vertgeq   & & & &   \vertgeq  & & \\
     u_2^{(2)} & \leq & u_2^{(3)} & \leq & \dots & \leq & u_2^{(2k-1)} &&\\
   \vdots & \vdots & \vdots & \vdots &\iddots  &  && &   \\
   u_{2k-1}^{(2k-1)} && & &&&&&
  \end{bmatrix}
 \frac{\Delta(u_1^{(2)}, u_2^{(2)})}{1!}d^{2k^2-k-1} u.
\end{align*}
We proceed inductively and get
\begin{align*}
 \gamma_{d_k,2}^O(c)=&2\int_{[0,1]^{\frac{3}{2}k^2-\frac{k}{2}}} \delta_c(u_1^{(k)}+\cdots+u_k^{(k)}) \delta_{c}(u_1^{(2k-1)}+u_3^{(2k-1)}+\cdots+u_{2k-1}^{(2k-1)})\\
&\times \mathds{1}   \begin{bmatrix}
    &   & u_1^{(k)} & \leq &   \dots &\leq & u_1^{(2k-1)}\\
    &   \iddots &  &   & & \iddots& \\
     u_k^{(k)} & \leq & \dots & \leq & u_k^{(2k-1)} &\\
   \vdots & &\iddots  & &  &&   \\
   u_{2k-1}^{(2k-1)} && & &&&
  \end{bmatrix}
\frac{\Delta(u_1^{(k)},u_2^{(k)},\dots,u_k^{(k)})}{1!2!\cdots (k-1)!} d^{\frac{3}{2}k^2-\frac{k}{2}} u.
\end{align*}

\end{proof}

\section{The determinant expression in the orthogonal case} \label{sec:determinant-orthogonal}
As in the symplectic case, consider in full generality
\[I_{d_k,\ell}^O(n;N):=\int_{\mathrm{O}(2N+1)} \Big|\sum_{\substack{j_1+\cdots+j_k=n\\0\leq j_1,\dots,j_k \leq 2N+1}}\mathrm{Sc}_{j_1}(U)\cdots \mathrm{Sc}_{j_k}(U)\Big|^\ell \mathrm{d}U,\]
and write
\[I_{d_k,\ell}^O(n;N) \sim \gamma_{d_k,\ell}^O(c) (2N+1)^m,\]
where $m$ is the appropriate exponent of $N$ such that the above asymptotic formula exists. 

Similar formulas for $\mathrm{O}(2N)$ and  $\mathrm{SO}(2N)$ with $\ell=1$  were studied in detail by Medjedovic \cite{Andy}. 
Recall that
\begin{equation}\label{eq:geno2}
\int_{\mathrm{O}(2N+1)} \det(1+Ux)^k\det(1+Uy)^k \mathrm{d}U=\sum_{m,n=0}^{2Nk}x^m y^nJ_{d_k,2}^O(m,n;N),
\end{equation}and we are interested in the diagonal terms 
\begin{equation*}
I_{d_k,2}^O(n;N):=J_{d_k,2}^O(n,n;N).
\end{equation*}

From (81) and (83) in \cite{Bump-Gamburd},
\begin{align}
\int_{\mathrm{SO}(2N+1)} \prod_{j=1}^r \det(1-Ux_j) \mathrm{d}U=&\frac{1}{\Delta(x_1,\dots, x_r)}\prod_{1\leq i<j\leq r}\frac{1}{1-x_ix_j}\det_{1\leq i, j \leq r} \left[ x_{i}^{2N+2r-j}-x_{i}^{j-1}\right]. \label{eq:detO}
\end{align}
Combining equations \eqref{eq:trulyfinalsoso},  \eqref{eq:geno2}, and \eqref{eq:detO}, yields
\begin{align*}
&\sum_{m,n=0}^{(2N+1)k}x^m y^nJ_{d_k,2}^O(m,n;N)\\
=& \frac{1}{(1-x^2)^{\binom{k}{2}}(1-y^2)^{\binom{k}{2}}(1-xy)^{k^2}}\left.\frac{\det_{1\leq i, j \leq 2k} \left[x_{i}^{2N+4k-j}- x_{i}^{j-1}\right]}{\Delta(x_1,\dots, x_{2k})} \right|_{(x,\dots,x,y,\dots,y)}\\
&+\frac{1}{(1-x^2)^{\binom{k}{2}}(1-y^2)^{\binom{k}{2}}(1-xy)^{k^2}}\left.\frac{\det_{1\leq i, j \leq 2k} \left[x_{i}^{2N+4k-j}- x_{i}^{j-1}\right]}{\Delta(x_1,\dots, x_{2k})} \right|_{(-x,\dots,-x,-y,\dots,-y)}\\
\end{align*}
Applying Lemma \ref{lem:superAndy}, we obtain
\begin{align}\nonumber 
&\sum_{m,n=0}^{(2N+1)k}x^m y^nJ_{d_k,2}^O(m,n;N)\nonumber\\
 =& \frac{2}{(1-x^2)^{\binom{k}{2}}(1-y^2)^{\binom{k}{2}}(1-xy)^{k^2}(y-x)^{k^2}}
\det_{\substack{1\leq i_1,i_2 \leq k\\1\leq j \leq 2k}}\left[\begin{array}{c}\binom{2N+4k-j}{i_1-1} x^{2N+4k-j-i_1+1}-\binom{j-1}{i_1-1} x^{j-i_1}\\ \\ \binom{2N+4k-j}{i_2-1} y^{2N+4k-j-i_2+1}- \binom{j-1}{i_2-1} y^{j-i_2}\end{array}\right].
\label{eq:deto}
\end{align}
Proposition \ref{prop:superVivian} implies that the right-hand side of \eqref{eq:deto} is at least a Taylor series. 

As in the symplectic case, we have a functional equation for the integral $I_{d_k,2}^O(n;N)$. 
 \begin{lem} 
 For $0\leq n \leq (2N+1)k$ the following functional equation holds:
  \[I_{d_k,2}^O(n;N)=I_{d_k,2}^O((2N+1)k-n;N).\]
 \end{lem}
 \begin{proof}
  The characteristic polynomial of a matrix in $\mathrm{SO}(2N+1)$ satisfies 
  \[\det(I+Ux)=x^{2N+1} \det(I+Ux^{-1}).\]
    Then by equation \eqref{eq:geno},
 \[\sum_{m,n=0}^{(2N+1)k}x^m y^nJ_{d_k,2}^O(m,n;N)=x^{(2N+1)k}y^{(2N+1)k} \sum_{m,n=0}^{(2N+1)k}x^{-m} y^{-n}J_{d_k,2}^O(m,n;N).\]
 Change variables $m\rightarrow (2N+1)k-m$, $n\rightarrow (2N+1)k-n$ to obtain
 \[J_{d_k,2}^S(m,n;N)=J_{d_k,2}^S((2N+1)k-m,(2N+1)k-n;N),\]
 and deduce the result from setting $m=n$. 
 \end{proof}

\begin{prop}\label{prop:nless2No} 
 For $n\leq N+\frac{k}{2}$,
 \[I_{d_k,2}^O(n;N)=\frac{2}{G(1+k)}\sum_{\substack{\ell=0\\\ell\equiv n \pmod{2}}}^n\binom{\frac{n-\ell}{2}+\binom{k}{2}-1}{\binom{k}{2}-1}^2\binom{\ell +k^2-1}{k^2-1}.\]
 Moreover, $I_{d_k,2}^O (n;N)$ is a quasi-polynomial in $n$ of degree $2k^2-k-2$ (if $n\leq N+\frac{k}{2}$). 
\end{prop}

\begin{proof}
Looking at the right-hand side of equation  \eqref{eq:deto}, we must take exclusively terms of the form 
 $\binom{j-1}{i-1} x^{j-i}$ and similarly with $y$. In other words, the $I_{d_k,2}^O$ comes exclusively from the diagonal terms of 
 \begin{align*} \frac{2}{(1-x^2)^{\binom{k}{2}}(1-y^2)^{\binom{k}{2}}(1-xy)^{k^2}(y-x)^{k^2}}
\det_{\substack{1\leq i_1,i_2 \leq k\\1\leq j \leq 2k}}\left[\begin{array}{c}\binom{j-1}{i_1-1} x^{j-i_1}\\ \\ \binom{j-1}{i_2-1} y^{j-i_2}\end{array}\right].
\end{align*}
 Lemma \ref{lem:y-xk2} implies that we must consider the diagonal terms in 
 \begin{align*}
& \frac{2}{G(1+k)(1-x^2)^{\binom{k}{2}}(1-y^2)^{\binom{k}{2}}(1-xy)^{k^2}} \\=& \frac{2}{G(1+k)} \sum_{\ell_1=0}^\infty \binom{\ell_1+\binom{k}{2}-1}{\binom{k}{2}-1}x^{2\ell_1} \sum_{\ell_2=0}^\infty \binom{\ell_2+\binom{k}{2}-1}{\binom{k}{2}-1}y^{2\ell_2} \sum_{m=0}^\infty \binom{m+k^2-1}{k^2-1}(xy)^m.
\end{align*}
Then
\begin{align*}
 I_{d_k,2}^O(n;N)=\frac{2}{G(1+k)}\sum_{\ell=0}^{\lfloor\frac{n}{2}\rfloor}\binom{\ell+\binom{k}{2}-1}{\binom{k}{2}-1}^2\binom{n-2\ell +k^2-1}{k^2-1}.
\end{align*}
This is a quasi-polynomial in $n$. Each term in the above sum has total degree $2\left(\binom{k}{2}-1\right)+k^2-1=2k^2-k-3$ in the variable $\ell$. By summing over $\ell$, each term leads to a polynomial of degree $2k^2-k-2$ on $n$. Moreover, the coefficients are all positive, which guarantees that there is no cancelation that could lower the final degree. 
 
\end{proof}

\subsection{Some low degree cases}
As in the symplectic case, Proposition \ref{prop:nless2No}  can be used to compute $I_{d_k,2}^O(n;N)$ for low values of $k$. For $k=1$, $I_{d_1,2}^O(n;N)=0.$

For $k=2$,
\begin{align*}
I_{d_2,2}^O(n;N)=&2\sum_{\substack{\ell=0\\\ell\equiv n \pmod{2}}}^{n}\binom{\ell+3}{3}\\
=&    2                \left \lfloor\frac{((n+3)^2-1)((n+3)^2-3)}{48}\right\rfloor\\
=&\frac{1}{48}( 2n^4 + 24n^3 + 100n^2 + 168n + 93)+\frac{(-1)^n}{16}.
\end{align*}

This leads to 
\[I_{d_2,2}^O(n;N) \sim \gamma_{d_2,2}^O(c)(2N+1)^4,\]
where 
  \[\gamma_{d_2,2}^O(c)=\begin{cases}
             \frac{c^4}{24}  & 0\leq c\leq \frac{1}{2},\\
             \frac{(2-c)^4}{24}  & \frac{3}{2}\leq c \leq 2.
             \end{cases}\]
 \section{A complex analysis approach in the orthogonal case} \label{sec:complex-orthogonal}
The goal of this section is to further understand the asymptotic coefficient $\gamma_{d_k,2}^O(c)$ by studying the integral \eqref{eq:geno} with tools from complex analysis.

By applying equation (3.35) from \cite[Lemma 4.3]{CFKRS-autocorrelation} to equation (81) from \cite{Bump-Gamburd}, we obtain 
 \begin{align*}
\int_{\mathrm{SO}(2N+1)} \prod_{j=1}^r\det(I-Ue^{-\alpha_j}) \mathrm{d}U =& \frac{(-1)^{r(r-1)/2}2^{r}}{(2\pi i)^r r!} e^{-\left(N+\frac{1}{2}\right)\sum_{j=1}^r \alpha_j} \oint \cdots \oint \prod_{1\leq \ell < m \leq r} (1-e^{-z_m-z_\ell})^{-1}\\
  &\times \frac{\Delta(z_1^2,\dots,z_r^2)^2 \prod_{j=1}^r \alpha_j e^{\left(N+\frac{1}{2}\right)\sum_{j=1}^r z_j} \mathrm{d}z_1\cdots \mathrm{d}z_r}{\prod_{i=1}^r \prod_{j=1}^r (z_j-\alpha_i)(z_j+\alpha_i)}.
 \end{align*}
Since we work with the full orthogonal group, we need to consider
\begin{align}\label{eq:O-SO}
\int_{\mathrm{O}(2N+1)} \prod_{j=1}^k\det(I-Ue^{-\alpha_j}) \mathrm{d}U=&\int_{\mathrm{SO}(2N+1)} \prod_{j=1}^k \det(I-Ue^{-\alpha_j}) \mathrm{d}U+\int_{\mathrm{SO}(2N+1)} \prod_{j=1}^k \det(I+Ue^{-\alpha_j}) \mathrm{d}U.
\end{align}
This leads to 
\begin{align*}
\int_{\mathrm{O}(2N+1)} \prod_{j=1}^{2k}\det(I+Ux_j) \mathrm{d}U=&\int_{\mathrm{SO}(2N+1)} \prod_{j=1}^{2k} \det(I-Ux_j) \mathrm{d}U+\int_{\mathrm{SO}(2N+1)} \prod_{j=1}^{2k} \det(I+Ux_j) \mathrm{d}U\\
=& \frac{(-1)^{k(2k-1)}2^{2k}}{(2\pi i)^{2k} (2k)!} \prod_{j=1}^{2k} x_j^{N+\frac{1}{2}} \oint \cdots \oint \prod_{1\leq \ell < m \leq 2k} (1-e^{-z_m-z_\ell})^{-1}\\
  &\times \frac{\Delta(z_1^2,\dots,z_{2k}^2)^2 \prod_{j=1}^{2k} (-\log(-x_j))e^{\left(N+\frac{1}{2}\right)\sum_{j=1}^{2k} z_j} \mathrm{d}z_1\cdots \mathrm{d}z_{2k}}{\prod_{i=1}^{2k} \prod_{j=1}^{2k} (z_j-\log(-x_i))(z_j+\log(-x_i))}\\
  &+(-1)^{k(2N+1)}\frac{(-1)^{k(2k-1)}2^{2k}}{(2\pi i)^{2k} (2k)!} \prod_{j=1}^{2k} x_j^{N+\frac{1}{2}} \oint \cdots \oint \prod_{1\leq \ell < m \leq 2k} (1-e^{-z_m-z_\ell})^{-1}\\
  &\times \frac{\Delta(z_1^2,\dots,z_{2k}^2)^2 \prod_{j=1}^{2k} (-\log(x_j))e^{\left(N+\frac{1}{2}\right)\sum_{j=1}^{2k} z_j} \mathrm{d}z_1\cdots \mathrm{d}z_{2k}}{\prod_{i=1}^{2k} \prod_{j=1}^{2k} (z_j-\log(x_i))(z_j+\log(x_i))}.
\end{align*}
By \eqref{eq:geno}, we are interested in the diagonal coefficients of the polynomial 
\begin{align*}
P_k(x,y):=&\int_{\mathrm{O}(2N+1)} \det(I-Ux)^k \det(I-Uy)^k\mathrm{d}U\\
=& \frac{(-1)^{k(2k-1)}2^{2k}}{(2\pi i)^{2k} (2k)!} x^{k\left(N+\frac{1}{2}\right)} y^{k\left(N+\frac{1}{2}\right)} \oint \cdots \oint \prod_{1\leq \ell < m \leq 2k} (1-e^{-z_m-z_\ell})^{-1}\\
  &\times \frac{\Delta(z_1^2,\dots,z_{2k}^2)^2 \log^kx \log^k y e^{\left(N+\frac{1}{2}\right)\sum_{j=1}^{2k} z_j} \mathrm{d}z_1\cdots \mathrm{d}z_{2k}}{ \prod_{j=1}^{2k} (z_j-\log(x))^k(z_j+\log(x))^k(z_j-\log(y))^k(z_j+\log(y))^k}\\
  &+(-1)^{k(2N+1)}\frac{(-1)^{k(2k-1)}2^{2k}}{(2\pi i)^{2k} (2k)!} 
  x^{k\left(N+\frac{1}{2}\right)} y^{k\left(N+\frac{1}{2}\right)} \oint \cdots \oint \prod_{1\leq \ell < m \leq 2k} (1-e^{-z_m-z_\ell})^{-1}\\
  &\times \frac{\Delta(z_1^2,\dots,z_{2k}^2)^2 \log^k(-x) \log^k(-y) e^{\left(N+\frac{1}{2}\right)\sum_{j=1}^{2k} z_j} \mathrm{d}z_1\cdots \mathrm{d}z_{2k}}{
  \prod_{j=1}^{2k} (z_j-\log(-x))^k(z_j+\log(-x))^k(z_j-\log(-y))^k(z_j+\log(-y))^k}.
\end{align*}
In order to compute the first integral above, first shrink the contour of the integrals into small circles centered at $\pm \log(x)$  and $\pm \log(y)$.  More precisely, let  $\epsilon_j=\log x $ or $\log y$ and $\delta_j=\pm 1$ for $j=1,\dots, 2k$. 
Consider 
\begin{align}\label{eq:pkdelta-o}P_k(x,y, \delta_1\epsilon_1,\dots \delta_{2k}\epsilon_{2k}):=&
 \oint \cdots \oint \prod_{1\leq \ell < m \leq 2k} (1-e^{-z_m-z_\ell})^{-1}\nonumber\\
  &\times \frac{\Delta(z_1^2,\dots,z_{2k}^2)^2 \log^kx \log^k y e^{\left(N+\frac{1}{2}\right)\sum_{j=1}^{2k} z_j} \mathrm{d}z_1\cdots \mathrm{d}z_{2k}}{ \prod_{j=1}^{2k} (z_j-\log(x))^k(z_j+\log(x))^k(z_j-\log(y))^k(z_j+\log(y))^k},
\end{align}
where the $j$th integral is around a small circle centered at $\delta_j\epsilon_j$ for $j=1,\dots, 2k$. This gives
\begin{align}\label{eq:Pksumo}
 P_k(x,y):= &\frac{2^{2k}}{(2\pi i)^{2k} (2k)!} x^{k\left(N+\frac{1}{2}\right)} y^{k\left(N+\frac{1}{2}\right)}\nonumber \\
 &\times \sum_{\delta_j, \epsilon_j} \left[(-1)^{k}P_k(x,y,\delta_1\epsilon_1,\dots \delta_{2k}\epsilon_{2k})+P_k(-x,-y,\delta_1\epsilon_1,\dots \delta_{2k}\epsilon_{2k})\right].
 \end{align}

As in the symplectic case, we have that the only non-zero terms in the above equation arise from those where $\epsilon_j=\log x$ for exactly half of the $j$. The proof of the next statement follows the same lines as the proof of Lemma \ref{lem:4.11}, and we omit it.  \begin{lem}\label{lem:4.11-o}(Equivalent of \cite[Lemma 4.11]{KR3}) Let $P_k(x,y,\delta_1\epsilon_1,\dots \delta_{2k}\epsilon_{2k})$  be defined by \eqref{eq:pkdelta-o}. Then 
  $P_k(x,y,\delta_1\epsilon_1,\dots \delta_{2k}\epsilon_{2k})=0$  unless exactly half of the $\epsilon_j$'s equal $\log x$. 
\end{lem}
 
We proceed to compute $P_k(x,y,\delta_1\epsilon_1,\dots \delta_{2k}\epsilon_{2k})$. First make the change of variables $z_j=\delta_j\epsilon_j+\frac{v_j}{N+\frac{1}{2}}$. 
By Lemma \ref{lem:4.11-o} and ignoring terms of order $\frac{1}{N+\frac{1}{2}}$ smaller, the integrand of $P_k(x,y,\delta_1\epsilon_1,\dots \delta_{2k}\epsilon_{2k})$ is asymptotic to 
\begin{align}\label{eq:page17-o}
  &2^{-2k} (-1)^{k^2}  \prod_{j=1}^{2k} \delta_j^{-k}x^{\left(N+\frac{1}{2}\right)\sum_{i=1}^k \delta_i} y^{\left(N+\frac{1}{2}\right)\sum_{i=k+1}^{2k} \delta_i} \times \prod_{1\leq \ell < m \leq k} 
  \left(1- x^{-\delta_m-\delta_\ell}e^{\frac{-v_m-v_\ell}{N+\frac{1}{2}}}\right)^{-1} \nonumber \\
 &\prod_{1\leq\ell  \leq k<m \leq 2k}\left(1- x^{-\delta_\ell} y^{-\delta_m} e^{\frac{-v_m-v_\ell}{N+\frac{1}{2}}}\right)^{-1}\times
 \prod_{k< \ell < m \leq 2k}\left(1- y^{-\delta_m-\delta_\ell}e^{\frac{-v_m-v_\ell}{N+\frac{1}{2}}}\right)^{-1}\nonumber \\
  & \times\prod_{1\leq \ell < m \leq k} (\delta_\ell v_\ell -\delta_m v_m)^2  \prod_{k< \ell < m \leq 2k} (\delta_\ell v_\ell -\delta_m v_m)^2 \frac{e^{\sum_{j=1}^{2k} v_j} }{\prod_{j=1}^{2k} v_j^k}\mathrm{d}v_1\cdots \mathrm{d}v_{2k}.
  \end{align}
Let $a$ be the number of ``$+$'' signs among the $\delta_\ell$ corresponding to $x$ and $b$ be the number of ``$+$'' signs among the $\delta_m$ corresponding to  $y$.  Taking only the products containing a reciprocal of linear terms with $x$ and $y$ and expanding as in the symplectic case, the factors in \eqref{eq:page17-o} are asymptotic to 
 \begin{align}\label{eq:52star}
 &(-1)^{\binom{a}{2}} x^{a(a-1)}\prod_{\substack{1\leq \ell < m \leq k\\ \delta_\ell\not =\delta_m}}  \left(1- e^{\frac{-v_m-v_\ell}{N+\frac{1}{2}}}\right)^{-1} \prod_{\substack{1\leq \ell < m \leq k\\ \delta_\ell=\delta_m=1}} e^{\frac{v_m+v_\ell}{N+\frac{1}{2}}}  \prod_{\substack{1\leq \ell < m \leq k\\ \delta_\ell=\delta_m}} \left(1- x^2e^{\frac{\delta_m v_m+\delta_\ell v_\ell}{N+\frac{1}{2}}}\right)^{-1}\nonumber\\
 &\times (-1)^{\binom{b}{2}} y^{b(b-1)}\prod_{\substack{k< \ell < m \leq 2k\\ \delta_\ell\not =\delta_m}}  \left(1- e^{\frac{-v_m-v_\ell}{N+\frac{1}{2}}}\right)^{-1} \prod_{\substack{k< \ell < m \leq 2k\\ \delta_\ell=\delta_m=1}} e^{\frac{v_m+v_\ell}{N+\frac{1}{2}}}  \prod_{\substack{k<\ell < m \leq 2k\\ \delta_\ell=\delta_m}} \left(1- y^2e^{\frac{\delta_m v_m+\delta_\ell v_\ell}{N+\frac{1}{2}}}\right)^{-1}\nonumber \\
 &\times (-1)^{kb}x^{2ab-kb} y^{kb} \prod_{\substack{1\leq\ell  \leq k<m \leq 2k\\ \delta_m=1 }}
 e^{\frac{v_m+v_\ell}{N+\frac{1}{2}}}\prod_{\substack{1\leq\ell  \leq  k<m \leq 2k\\ \delta_\ell=\delta_m}}\left(1- x ye^{\frac{\delta_m v_m+\delta_\ell v_\ell}{N+\frac{1}{2}}}\right)^{-1}\nonumber \\
 & \times \prod_{\substack{1\leq\ell  \leq k<m \leq 2k\\ \delta_\ell=1, \delta_m=-1 }}\left(1- x^{-1} ye^{\frac{-v_m-v_\ell}{N+\frac{1}{2}}}\right)^{-1}\prod_{\substack{1\leq\ell  \leq k<m \leq 2k\\ \delta_\ell=-1, \delta_m=1 }} \left(1-x^{-1}ye^{\frac{v_m+v_\ell}{N+\frac{1}{2}}}\right)^{-1}.
\end{align}
Expanding and considering the asymptotics for $N$, \eqref{eq:52star} becomes
\begin{align*}
\sim &(-1)^{\binom{a+b}{2}+(k-a)b} x^{a(a+b-1)-(k-a)b}\frac{N^{a(k-a)}}{\prod_{\substack{1\leq \ell < m \leq k\\ \delta_\ell\not =\delta_m}} (v_m+v_\ell)}  y^{b(a+b-1)+(k-a)b}\frac{N^{b(k-b)}}{\prod_{\substack{k< \ell < m \leq 2k\\ \delta_\ell\not =\delta_m}} (v_m+v_\ell)} \\
&\times \sum_{n=0}^\infty \sum_{\substack{\sum_{ 1\leq \ell < m \leq k} \beta_{\ell,m}=n\\ \beta_{\ell,m} \geq 0,  \delta_\ell=\delta_m\\ \beta_{\ell,m}= 0, \delta_\ell\not =\delta_m }} x^{2n} \exp \Big(\sum_{\substack{1\leq \ell < m \leq k\\\delta_\ell=\delta_m}} \beta_{\ell,m}\frac{\delta_m v_m+\delta_\ell v_\ell}{N+\frac{1}{2}}\Big)\\ 
 &\times  \sum_{n=0}^\infty \sum_{\substack{\sum_{ k<\ell < m \leq 2k} \beta_{\ell,m}=n\\ \beta_{\ell,m} \geq 0,  \delta_\ell=\delta_m\\ \beta_{\ell,m}= 0, \delta_\ell\not =\delta_m }} y^{2n} \exp \Big(\sum_{\substack{k< \ell < m \leq 2k\\\delta_\ell=\delta_m}} \beta_{\ell,m}\frac{\delta_m v_m+\delta_\ell v_\ell}{N+\frac{1}{2}}\Big)\\ 
 &\times \sum_{n=0}^\infty \sum_{\substack{\sum_{1\leq \ell\leq k < m \leq 2k} \beta_{\ell,m}=n\\ \beta_{\ell,m} \geq 0,  \delta_\ell=\delta_m\\ \beta_{\ell,m}= 0, \delta_\ell\not =\delta_m }} x^ny^{n} \exp \Big(\sum_{\substack{1\leq \ell\leq k < m \leq 2k\\\delta_\ell=\delta_m}} \beta_{\ell,m}\frac{\delta_m v_m+\delta_\ell v_\ell}{N+\frac{1}{2}}\Big)\\ 
 &\times \sum_{n=0}^\infty \sum_{\substack{\sum_{ 1\leq \ell\leq k < m \leq 2k} \beta_{\ell,m}=n\\ \beta_{\ell,m} \geq 0, \delta_\ell \not = \delta_m\\\beta_{\ell,m}= 0, \delta_\ell = \delta_m}} x^{-n}y^{n} \exp \Big(\sum_{\substack{1\leq \ell\leq k < m \leq 2k\\ \delta_\ell\not = \delta_m}} \beta_{\ell,m}\delta_m\frac{v_m
 +v_\ell}{N+\frac{1}{2}}\Big).
\end{align*}

As in the symplectic case, the coefficient contributing to $x^{(2c -k)\left(N+\frac{1}{2}\right)}y^{(2c-k)  \left(N+\frac{1}{2}\right)}$ coming from the four infinite sums above is 
\begin{align}\label{eq:Riemannsum-o}
&\sideset{}{'}\sum_{\substack{(\beta_{\ell,m})\\ 1 \le \ell < m \le 2k}}\exp\Big(\frac{1}{N+\frac{1}{2}} \sideset{}{'}\sum_{\ell,m} \beta_{\ell,m} \delta_m (v_\ell + v_m)\Big),
\end{align}
where the sum (and by an abuse of notation, the inner sum as well) is taken over all $\beta_{\ell,m}$ that appear in the five sums above, that is to say, all pairs $1\le \ell < m \le 2k$ except for those where $1\le \ell < m \le k$ or $k < \ell < m \le 2k$ and $\delta_\ell \ne \delta_m$, and which satisfy the two constraints that
\begin{align*}
&2\sum_{\substack{1 \le \ell < m \le k \\ \delta_\ell = \delta_m}} \beta_{\ell,m} + \sum_{\substack{1 \le \ell \le k < m \le 2k \\ \delta_\ell = \delta_m}} \beta_{\ell,m} - \sum_{\substack{1 \le \ell \le k < m \le 2k \\ \delta_\ell \neq \delta_m}} \beta_{\ell,m}\nonumber \\&  = (2c -k +2a-k) \left(N+\frac{1}{2}\right) - a(a+b-1)+(k-a)b,
\end{align*}
and
\begin{align*}
&2\sum_{\substack{k < \ell < m \le 2k \\ \delta_\ell = \delta_m}} \beta_{\ell,m} + \sum_{\substack{1 \le \ell \le k < m \le 2k \\ \delta_\ell = \delta_m}} \beta_{\ell,m} - \sum_{\substack{1 \le \ell \le k < m \le 2k \\ \delta_\ell \neq \delta_m}} \beta_{\ell,m} \nonumber \\& = (2c -k+2b-k)  \left(N+\frac{1}{2}\right)- b(a+b-1)-(k-a)b,
\end{align*}
Now change variables via 
\[x_{\ell, m} =  \frac{\beta_{\ell, m}}{N+\tfrac 12} \]
As $N\rightarrow \infty$, the previous conditions become
\begin{align}\label{eq:xxconstraint-o}
&2\sum_{\substack{1 \le \ell < m \le k \\ \delta_\ell = \delta_m}} x_{\ell,m} + \sum_{\substack{1 \le \ell \le k < m \le 2k \\ \delta_\ell = \delta_m}} x_{\ell,m} - \sum_{\substack{1 \le \ell \le k < m \le 2k \\ \delta_\ell \neq \delta_m}} x_{\ell,m} = 2c+2a
\end{align}
and 
\begin{align}\label{eq:yyconstraint-o}
&2\sum_{\substack{k < \ell < m \le 2k \\ \delta_\ell = \delta_m}} x_{\ell,m} + \sum_{\substack{1 \le \ell \le k < m \le 2k \\ \delta_\ell = \delta_m}} x_{\ell,m} - \sum_{\substack{1 \le \ell \le k < m \le 2k \\ \delta_\ell \neq \delta_m}} x_{\ell,m} = 2c+2b,
\end{align}
where $x_{\ell,m}\geq 0$.  As in the symplectic case, interpret \eqref{eq:Riemannsum-o} as a Riemann sum, along the lines of \cite[Lemma 4.12]{KR3}. The sum is over $2k^2 - k-a(k-a)-b(k-b)$ variables with two constraints, so \eqref{eq:Riemannsum-o} is
\begin{align*}
&N^{2k^2 - k-a(k-a)-b(k-b) -2} \frac 1{N^{2k^2-k-a(k-a)-b(k-b)-2}}\sideset{}{'}\sum_{\substack{(\beta_{\ell,m})\\ 1 \le \ell < m \le 2k}} \exp\Big(\sideset{}{'}\sum_{\ell,m} \frac{\beta_{\ell,m}}{N} \delta_m (v_\ell + v_m)\Big) \\
&= N^{2k^2 - k-a(k-a)-b(k-b) - 2} \sideset{}{'}\iint \exp\Big(\sideset{}{'}\sum_{\ell,m} x_{\ell,m} \delta_m (v_\ell + v_m)\Big) \sideset{}{'}\prod_{\ell,m} \mathrm{d}x_{\ell,m}\\& + O(N^{2k^2 - k-a(k-a)-b(k-b) - 3}).
\end{align*}

Finally, the coefficient of $x^{(2c-k) \left(N+\frac{1}{2}\right)}y^{(2c-k) \left(N+\frac{1}{2}\right)}$ in \eqref{eq:page17-o} is asymptotic to
\begin{align*}
\sim & 2^{-2k} (-1)^{a(k+b)+k+\binom{a+b}{2}}N^{2k^2 - k- 2}  \oint \sideset{}{'}\iint \exp\Big(\sideset{}{'}\sum_{\ell,m} x_{\ell,m} \delta_m (v_\ell + v_m)\Big) \sideset{}{'}\prod_{\ell,m} \mathrm{d}x_{\ell,m}\\
  & \times\frac{\prod_{1\leq \ell < m \leq k} (\delta_\ell v_\ell -\delta_m v_m)^2  \prod_{k< \ell < m \leq 2k} (\delta_\ell v_\ell -\delta_m v_m)^2}{\prod_{\substack{1\leq \ell < m \leq k\\ \delta_\ell\not =\delta_m}} (v_m+v_\ell)\prod_{\substack{k< \ell < m \leq 2k\\ \delta_\ell\not =\delta_m}} (v_m+v_\ell)} \frac{e^{\sum_{j=1}^{2k} v_j} }{\prod_{j=1}^{2k} v_j^k}\mathrm{d}v_1\cdots \mathrm{d}v_{2k}.
\end{align*}
Define the  integral
\[J({\bf v}):=\iint_{\substack{x_{\ell,m}\geq 0\\\eqref{eq:xxconstraint-o}, \eqref{eq:yyconstraint-o}}} \exp\Big(\sideset{}{'}\sum_{\ell,m} x_{\ell,m} \delta_m (v_\ell + v_m)\Big) \sideset{}{'}\prod_{\ell,m} \mathrm{d}x_{\ell,m}.\]
As in the symplectic case, as ${\bf v}\rightarrow 0$,  $J({\bf v})$ becomes a polynomial in $c$ of degree $2k^2 - k-2-a(k-a)-b(k-b)$ (or identically 0). Proceeding with the integration as described in the symplectic case, and doing the same process for $P_k(-x,-y,\delta_1\epsilon_1,\dots \delta_{2k}\epsilon_{2k})$, we obtain that the main term in \eqref{eq:Pksumo} is given by 
\[\gamma_{d_k,2}^O(c) (2N+1)^{2k^2 - k- 2},\]
where
\[\gamma_{d_k,2}^O(c)=\sum_{\substack{0\leq b\leq c \\ 0 \le a \le 2c-b}} g_{a,b}^O(c),\]
and each $g_{a,b}^O(t)$ is a polynomial of degree $2k^2-k-2$. 

Thus $\gamma_{d_k,2}^O(c)$ is a piecewise polynomial function of degree at most $2k^2-k-2$, and this concludes the proof of Theorem \ref{thm:complexorthogonal}.

\section{The conjectures} \label{sec:conjectures}

In this section we explain how we have reached Conjecture \ref{conj:symp-square}  and formulate a conjecture for the Gaussian integers inspired by the setting considered by Rudnick and Waxman \cite{Rudnick-Waxman} as well as Theorem 1.2 from \cite{KuperbergLalin}. 

The integral \eqref{eq:intsym} can be rewritten as 
 \begin{equation} \label{eq:varsff}
\mathrm{Var}(\mathcal{S}^S_{d_k,n})
 \sim  \frac{q^n}{4} \gamma_{d_k,2}^S\left(\frac{n}{2g}\right) (2g)^{2k^2+k-2},
\end{equation}
under the condition $n\leq 2gk$, 
where 
\begin{equation}\label{eq:sdkff}\mathcal{S}^S_{d_k,n}(P):=\sum_{\substack{f \,\text{monic}, \deg(f)=n  \\f\equiv \square \pmod{P}\\P\nmid f}} d_k(f),\end{equation}
and $\deg(P)=2g+1$.

In \eqref{eq:varsff}, we wrote $\gamma_{d_k,2}^S$ as a function of $n/2g$ by factoring out $(2g)^{2k^2+k-2}$, where in Sections \ref{sec:symmetric-symplectic}, \ref{sec:determinant-symplectic}, and \ref{sec:complex-symplectic} we have shown that $2k^2+k-2$ is the degree of $\gamma_{d_k,2}^S(c)$. The study of the degree of $\gamma_{d_k,2}^S$ is what allows us to properly formulate Conjecture \ref{conj:symp-square}.

We are interested in obtaining a number field version of formula \eqref{eq:varsff} for 
\begin{equation}\label{eq:sdknf}
\mathcal{S}^S_{d_k;x}(p)=\sum_{\substack{m\leq x  \\n\equiv \square \pmod{p}\\p\nmid m}}d_k(m).
\end{equation}
In \eqref{eq:sdkff} $q^n$ represents the size of $f$, and it therefore corresponds to $x$, the size of $m$, in \eqref{eq:sdknf}.  
Similarly $q^{2g+1}$ represents the size of $P$ in \eqref{eq:sdkff}. To control the size of $p$, it is natural to consider a  condition of type $y \leq p \leq 2y$. Replacing $q^n$ by $x$, $n$ by $\log x$ and $2g$ by $\log y$, we reach formula \eqref{eq:conj}, except for the arithmetic factor $a_k^S(\mathcal S)$.  One should be able to deduce the arithmetic factor by following the work of Conrey, Farmer, Keating, Rubinstein, and Snaith \cite{CFKRS, CFKRS-autocorrelation} which gives a heuristic by comparing the number field setting with the random matrix case that arises from the function field setting. 


The variance of 
\[\mathcal{N}^S_{d_\ell,k,n}(v)=\sum_{\substack{f \in \mathcal{M}_n\\ f(0)\not =0\\U(f)\in \mathrm{Sect}(v,k)}} d_\ell(f)\]
was considered in \cite{KuperbergLalin}. It is given by equation \eqref{thm:sym-RW}, which can be written as 
\[\mathrm{Var}(\mathcal N^S_{d_\ell,k,n}) \sim  \frac{q^{n}}{q^\kappa}
\gamma_{d_\ell,2}^S\left(\frac{n}{2\kappa}\right)  (2\kappa)^{2\ell^2+\ell-2}.\]
Based on the Gaussian integers problem described in the introduction, we get Conjecture \ref{conj:symp-rudnickwaxman} by replacing $q^n$ by $x$ and $q^\kappa$ by $K$.

Theorem \ref{thm:ortho-intro} can be rewritten as
\[\mathrm{Var}(\mathcal N^O_{d_\ell,k,n}) \sim  \frac{q^{n}}{4q^{\kappa}}
\gamma_{d_\ell,2}^O\left(\frac{n}{2\kappa }\right)  (2\kappa)^{2\ell^2-\ell-2},\]
under the condition $n\leq \ell (2\kappa-1)$, where 
 \[\mathcal{N}^O_{d_\ell,k,n}(v)=\sum_{\substack{f \in \mathcal{M}_n\\ f(0)\not =0\\U(f)\in \mathrm{Sect}(v
 ,k)}} d_\ell(f)\left(\frac{1 + \chi_2(f)}{2}\right).\]

 Recall that for a prime $P(S)\in \F_q[S]$, $P(0)$ is a square in $\F_q$ if and only if it has an expression such as \eqref{eq:norm}, so the character $\chi_2$ is effectively detecting whether or not $-T$ is a square mod $P$. It is possible define an analogous character for the Gaussian integers $\mathbb Z[i]$; for example,  
one can consider the quartic residue symbol 
 defined for $\pi$ prime in $\Z[i]$ such that  $(\pi)\not = (1+i)$ by 
\[\chi_2((\pi)):= \left[\frac{2i}{\pi}\right]= (2i)^{\frac{N(\pi)-1}{4}}\pmod{\pi}.\]
(See  \cite[Chapter 9.8]{IrelandRosen}.)

However, the relevant family of $L$-functions $L(s,\chi_2\psi^n)$, where $\psi$ is the Hecke grossencharacter
\[\psi(\mathfrak{a}):=e^{i\theta_\mathfrak{a}} = \frac{(a+bi)^4}{(a^2+b^2)^2},\]
is also a symplectic family, with no orthogonal symmetry. Indeed, (and somewhat surprisingly), as noted by Katz \cite{Katz}, twisting by the quadratic character in $\mathbb F_q[t]$ does not (yet) seem to correspond to the function field analogue of any classical number-theoretic problem.


 Orthogonal families of rational $L$-functions are less common than symplectic or unitary families, but some are known. In the same way that we predict that $\gamma_{d_k,2}^S(c)$ appears in the variance of rational arithmetic problems exhibiting symplectic behavior, we predict that $\gamma_{d_k,2}^O(c)$ appears in problems related to orthogonal families. For example, Iwaniec, Luo, and Sarnak \cite{IwaniecLuoSarnak} mention the family of holomorphic cusp forms of weight $k$ which are newforms of level $N$. Separating the spaces of the newforms with root number $\varepsilon_f=1$ and those with root number $\varepsilon_f=-1$, the statistics associated to the $L$-functions of those families give $\mathrm{SO}(\text{even})$ for the space of $\varepsilon_f=1$, and $\mathrm{SO}(\text{odd})$ for the space of $\varepsilon_f=-1$. Since $\mathrm{O}(2m+1) = \mathrm{SO}(2m+1) \cup (-1) \cdot \mathrm{SO}(2m+1)$, we can directly connect the $\mathrm{SO}(2m+1)$ and the $\mathrm{O}(2m+1)$ case. In fact, we already used this in Section \ref{sec:complex-orthogonal}, as our starting point was equation \eqref{eq:O-SO}.
Other cases of orthogonal symmetry are described by Conrey and Farmer \cite{ConreyFarmer} (pages 885-886).

\bibliographystyle{amsalpha}

\bibliography{Bibliography}

\end{document}